\documentclass[a4paper,10pt]{article}
\usepackage[utf8]{inputenc}
\usepackage{setspace}

 \usepackage[lmargin=3.5cm, rmargin=3.5cm]{geometry}
\usepackage{mathtools}
\usepackage{commath}
\usepackage{amsmath}
\usepackage{amsthm}
\usepackage{amssymb}
\usepackage{amsfonts}
\usepackage{graphicx}
\usepackage{color}
\usepackage{calc}
\usepackage{IEEEtrantools}
\usepackage[percent]{overpic}
\usepackage{xcolor}
\usepackage{float}
\usepackage{enumerate}
\usepackage{tikz-cd}
\usetikzlibrary{decorations.pathmorphing} \tikzset{snake it/.style={decorate, decoration=snake}}
\usepackage{lmodern}

\newcommand{\bracedincludegraphics}[2][]{%
  \sbox0{$\vcenter{\hbox{\includegraphics[#1]{#2}}}$}%
  \left\lbrace
    \vphantom{\copy0}
  \right.\kern-\nulldelimiterspace
  \underbrace{\vrule width0pt depth \dimexpr\dp0 + .3ex\relax\box0}}

\newtheorem{theorem}{Theorem}
\newtheorem{lemma}[theorem]{Lemma}
\newtheorem{proposition}[theorem]{Proposition}

\theoremstyle{definition}

\title{The invariant measure of PushASEP with a wall and point-to-line last passage percolation}
\author{Will FitzGerald\thanks{w.fitzgerald@sussex.ac.uk}}

\begin{document}

\maketitle

\abstract{
We consider an interacting particle system on the lattice involving pushing and blocking interactions, called PushASEP, in the presence of a wall at the origin.
We show that the invariant measure of this system 
is equal in distribution to a vector of point-to-line last passage percolation times in a random geometrically
distributed
environment. 
The largest co-ordinates in both of these vectors are equal in distribution to the all-time supremum of a non-colliding random walk. 
}

\paragraph{2020 Mathematics Subjects Classifications.} 60K35, 60C05, 60J45

\section{Introduction}
The last two decades have seen remarkable progress in the study of random interface growth, interacting 
particle systems and random polymers within the Kardar-Parisi-Zhang (KPZ) universality class through the
identification of deep connections between probability, combinatorics, symmetric functions, queueing theory, random matrices and quantum integrable systems. The greatest progress has been made with narrow-wedge initial data 
(for example see
\cite{baik_deift_johansson, baryshnikov, borodin_corwin_macdonald, cosz, johansson_2, o_connell2012, prahofer_spohn, tracy_widom_asep}) and
there are substantial differences in the case of flat initial data, see  \cite{baik_rains, bisi_zygouras, bfps, ferrari, remenik_nguyen, sasamoto}.

The purpose of this paper is to prove multi-dimensional identities in law
between different models in the KPZ universality class with flat initial data.
These are closely related to identities involving reflected Brownian motions and point-to-line last passage percolation with exponential data proved recently in \cite{FW}. The results of this paper, together with \cite{FW}, suggests
the possibility that there may be more identities of this form and deeper algebraic reasons for why they hold. 

On the one hand, these identities involve an interacting particle system called PushASEP (introduced in \cite{borodin2008}) in the presence of an additional wall at the origin. This is a continuous-time Markov chain $(Y_1(t), \ldots, Y_n(t))_{t \geq 0}$
taking values in $W^n_{\geq 0} = \{(y_1, \ldots, y_n) : 0 \leq y_1 \ldots \leq y_n \text{ and } y_i \in \mathbb{Z}\}$ and
with the following evolution depending on $2n$ independent exponential clocks. Throughout we refer to the $i$-th co-ordinate as the $i$-th particle.
At rate $v_i$, the right-clock of the $i$-th particle rings and the $i$-th particle jumps to the right. All particles which have 
(before the jump of the $i$-th particle) a position equal to the $i$-th particle position and an index greater than or equal to $i$ are 
\emph{pushed} by one step to the right. At rate $v_i^{-1}$ the left-clock of the $i$-th particle rings and if the $i$-th particle has a position strictly larger than both the $(i-1)$-th particle and zero then
the $i$-th particle jumps by one step to the left; if not this jump is suppressed. In summary, particles \emph{push} particles with higher indices and are \emph{blocked} by particles with lower indices and a wall at the origin.


A second viewpoint is to relate the top particle in PushASEP with a wall 
to the top particle in an ordered (or non-colliding process),  see Proposition \ref{two_sided_intertwining} and related statements in \cite{baryshnikov, biane2005, o_connell2012, warren}.
Let $(Z_1^{(v_n)}(t), \ldots, Z_n^{(v_1)}(t))_{t \geq 0}$ be a multi-dimensional continuous-time random walk where $Z_i^{(v_{n-i+1})}$ jumps to the right with rate $v_{n-i+1}$ and to the left with 
rate $v_{n-i+1}^{-1}$. We construct from this an ordered process
$(Z^{\dagger}_1(t), \ldots, Z^{\dagger}_n(t))_{t \geq 0}$ by a Doob $h$-transform, see Section \ref{Harmonic_Functions}. 
In the case $0 < v_n < \ldots < v_1$, 
this is given by conditioning $(Z_1^{(v_n)}, \ldots, Z_n^{(v_1)})$ 
on the event of positive probability that $Z_1^{(v_n)} \leq \ldots \leq Z_n^{(v_1)}$.

The other side of the identities we prove, involve point-to-line last passage percolation times. 
Let $\Pi_n^{\text{flat}}$ denote the set of all directed (up and right) nearest neighbour paths from the point $(k, l)$
to the line $\{(i, j): i+j = n+1\}$ and 
let 
\begin{equation}
\label{lpp_defn}
G(k, l) = \max_{\pi \in \Pi_n^{\text{flat}}(k, l)} \sum_{(i, j) \in \pi} g_{ij}
\end{equation}
where $g_{ij}$ are an independent collection of geometric random variables with parameter $1- v_i v_{n-j+1}$
indexed by 
$\{(i, j) : i, j \in \mathbb{Z}_{\geq 1} \text{ and } i + j \leq n+1\}$ and
with $0 < v_i < 1$ for each $i = 1, \ldots, n$.
The geometric random variables are defined as $P(g_{ij} = k) = (1-v_i v_{n-j+1}) (v_i v_{n-j+1})^{k}$ for all $k \geq 0$.


\begin{theorem}
\label{top_particle_theorem}
Let $n \geq 1$ and suppose $0 < v_1, \ldots, v_n < 1$ and let $Y_n^*$ be distributed according to the top particle of PushASEP with a wall 
in its invariant measure, let $Z_n^{\dagger}$ be the top particle in the ordered random walk above, see also \eqref{non_colliding_defn}, and $G(1, 1)$ be the point-to-line last passage percolation time defined by \eqref{lpp_defn}. Then 
\begin{equation*}
Y_n^*  \stackrel{d}{=} \sup_{t \geq 0} Z_n^{\dagger}(t)  \stackrel{d}{=} G(1, 1).
\end{equation*}
\end{theorem}

The first identity in Theorem \ref{top_particle_theorem} follows from two representations for $Y_n^*$ 
and $\sup_{t \geq 0} Z_n^{\dagger}(t)$ as point-to-line last passage percolation times in a random
environment constructed from Poisson point processes. The equality in law then follows from
a time reversal argument. 

The main content of Theorem 
\ref{top_particle_theorem} is that either of these random variables is equal in distribution to a point-to-line last 
passage percolation time. 
This can be proven in two ways. The first method is to calculate the distribution function of 
$\sup_{t \geq 0} Z_n^{\dagger}(t)$ 
by relating the problem to conditioning a multi-dimensional random walk to stay in a Weyl chamber of type C given that 
it remains in a Weyl chamber of type A. This gives the distribution function of $\sup_{t \geq 0} Z_n^{\dagger}(t)$ as proportional to a symplectic Schur function divided by a Schur function. This can be identified as a known expression 
for the distribution function of point-to-line last passage percolation in a geometric environment from \cite{Bisi_Zygouras2}. This proof 
of Theorem \ref{top_particle_theorem} is given in Section \ref{Harmonic_Functions}.

The second method of proof is to view Theorem \ref{top_particle_theorem} as an equality of the marginal distributions of the largest co-ordinates in a multi-dimensional identity in law relating the whole invariant measure of PushASEP with a wall 
to a vector of point-to-line last passage percolation times. This leads to our main result. 
\begin{theorem}
\label{main_theorem}
Let $n \geq 1$ and suppose $0 < v_1, \ldots, v_n < 1$. 
Let $(Y_1^*, \ldots, Y_n^*)$ be distributed according to the invariant measure of PushASEP with a wall and let 
$(G(1, n), \ldots, G(1, 1))$ be a vector of point-to-line last passage percolation times defined in \eqref{lpp_defn}.
Then
\begin{equation*}
(Y_1^*, \ldots, Y_n^*)  \stackrel{d}{=} (G(1, n), \ldots, G(1, 1)).
\end{equation*}
\end{theorem}

We give two proofs of Theorem \ref{main_theorem}. In the first proof, we prove in Section \ref{push_asep} a formula for the transition probability of PushASEP 
with a wall, following the method of \cite{borodin2008}. From this we obtain an expression for the
probability mass function of 
$(Y_1^*, \ldots, Y_n^*)$ in Proposition \ref{Invariant_measure}. 
In Section \ref{point_to_line}, we use an interpretation of last passage percolation as a 
discrete-time Markov chain, with a sequential update rule for particle positions, which has explicit 
determinantal transition probabilities given in \cite{dieker2008}. In order to find the distribution of a vector of \emph{point-to-line last passage percolation times}, we 
use the update rule of this discrete-time Markov chain while adding in a new particle at the origin after each time step. 
In such a way we can find an explicit probability mass function for $(G(1, n), \ldots, G(1, 1))$ which agrees with 
$(Y_1^*, \ldots, Y_n^*)$ and gives our first proof of Theorem \ref{main_theorem}. 

The second proof of Theorem \ref{main_theorem} is to obtain this multi-dimensional equality in law as a marginal equality of a larger identity in law. We give this proof in Section \ref{dynamic_reversibility}.
In particular, we construct a multi-dimensional Markov process involving pushing and blocking 
interaction which has (i) an invariant measure given by $\{G(i, j) : i + j \leq n+1\}$ and (ii) 
a certain marginal given by PushASEP with a wall.
Moreover, the process we construct is \emph{dynamically reversible}. This notion has appeared in the queueing literature \cite{kelly} and 
means that a process started in stationarity has the same distribution when run forwards and backwards in time 
\emph{up to a 
relabelling of the co-ordinates}. Dynamical reversibility leads to a convenient way of finding an invariant measure and can be used to deduce further properties of PushASEP with a wall. In particular, when started in stationarity the top particle of PushASEP with a wall evolves as a non-Markovian process with the same distribution when run forwards and backwards 
in time. This is a property shared by the $\text{Airy}_1$ process and it is natural to expect that the 
top particle in PushASEP with a wall
run in stationarity converges to the $\text{Airy}_1$ process.

We end the introduction by comparing with the results on PushASEP in Borodin and Ferrari \cite{borodin2008}.
When started from a step or periodic initial condition \cite{borodin2008}
prove that
the associated height function converges to the $\text{Airy}_2$ or $\text{Airy}_1$ process respectively 
(see also the seminal work \cite{bfps, sasamoto}). The choice of a periodic initial condition thus gives one way of accessing the KPZ universality class started from a flat interface. In this paper we 
instead impose a wall at the origin and consider the invariant measure of PushASEP with a wall. This makes a substantial difference to the analysis and unveils different connections within the KPZ universality class with flat initial data.

\section{Proof of Theorem \ref{top_particle_theorem}}
\label{Harmonic_Functions}

\subsection{The all-time supremum of a non-colliding process}
\label{suprema_non_colliding}

We start by defining Schur and symplectic Schur functions. It will be sufficient for our purposes to define them according to their Weyl character formulas and we only remark that they can also be defined as a sum over weighted 
Gelfand Tsetlin patterns and have a representation theoretic significance, see  \cite{Fulton_Harris}.
Let $W^n = \{(x_1, \ldots, x_n) \in \mathbb{Z}^n: x_1 \leq \ldots \leq x_n\}$, 
$W^n_{\geq 0} = \{(x_1, \ldots, x_n) \in \mathbb{Z}^n: 0 \leq x_1 \leq \ldots \leq x_n\}$
and 
$W^n_{\leq 0} = \{(x_1, \ldots, x_n) \in \mathbb{Z}^n: x_1 \leq \ldots \leq x_n \leq 0\}$.
For $x \in W^n$ we define the Schur function $S_x : \mathbb{R}^n \rightarrow \mathbb{R}$ by
\begin{equation}
\label{schur_defn}
S_{x}(v) = \frac{\text{det}(v_i^{x_j + j - 1})_{i, j = 1}^n}{\text{det}(v_i^{j-1})_{i, j = 1}^n}
\end{equation}
and for $x \in W^n_{\geq 0}$ we define the symplectic Schur function $\text{Sp}_x : \mathbb{R}^n_{> 0} \rightarrow \mathbb{R}$ 
by 
\begin{equation}
\label{symplectic_schur_defn}
\text{Sp}_{x}(v) = \frac{\text{det}\left(v_i^{x_j + j} - v_i^{-(x_j + j)}\right)_{i, j = 1}^n}{\text{det}(v_i^{j} - v_i^{-j})_{i, j = 1}^n}.
\end{equation}

Let $(Z_1^{(v_n)}(t), \ldots, Z_n^{(v_1)}(t))_{t \geq 0}$ denote a multi-dimensional continuous-time random walk 
started from $(x_1, \ldots, x_n)$ where each component is independent and $Z_i^{(v_{n-i+1})}$ jumps to the right at rate $v_{n-i+1}$ and to the left with 
rate $v_{n-i+1}^{-1}$. We define an ordered random walk
$(Z^{\dagger}_1(t), \ldots, Z^{\dagger}_n(t))_{t \geq 0}$ started from $x \in W^n$ as having a $Q$-matrix
given by a Doob $h$-transform: for $x \in W^n$ 
and $i = 1, \ldots, n$,
\begin{equation}
\label{non_colliding_defn}
Q_{Z^{\dagger}}(x, x \pm e_i) = \frac{S_{x \pm e_i}(v)}{S_{x}(v)} 1_{\{x \pm e_i \in W^n\}}.
\end{equation}
 This is a version of  $(Z_1^{(v_n)}(t), \ldots, Z_n^{(v_1)}(t))_{t \geq 0}$ with components conditioned to remain ordered as $Z_1 \leq \ldots \leq Z_n$. 
It is related to a non-colliding random walk with components conditioned to remain strictly ordered by a co-ordinate change; for more information on non-colliding random walks we refer to \cite{konig2005, konig2002, o_connell_2003}.

Define
$h_A : W^n \rightarrow \mathbb{R}$ by 
$h_A(x) = \prod_{i = 1}^n v_{n-i+1}^{-x_i} S_x(v)$
and define
$h_C : W^n_{\leq 0} \rightarrow \mathbb{R}$ by 
$h_C(x_1, \ldots, x_n) = \prod_{i = 1}^n v_{n-i+1}^{-x_i} \text{Sp}_{(-x_n, \ldots, -x_1)}(v)$.

\begin{proposition}
\label{two_sided_intertwining}
\begin{enumerate}[(i)]
\item $Q_{Z^\dagger}$ is a conservative $Q$-matrix. 
Equivalently, $h_A$  is harmonic for $(Z_1^{(v_n)}(t), \ldots, Z_n^{(v_1)}(t))_{t \geq 0}$ killed when it leaves
$W^n$. 
\item We have that
\begin{equation}
(Z_n^{\dagger}(t))_{t \geq 0}  \stackrel{d}{=} 
\left(\sup_{0 = t_0 \leq t_1 \leq \ldots \leq t_n = t} \sum_{i=1}^n (Z_i^{(v_{n-i+1})}(t_i) - Z_i^{(v_{n-i+1})}(t_{i-1}))\right)_{t \geq 0}.
\end{equation}
\end{enumerate}
\end{proposition}
This is a consequence of Theorem 5.10 in \cite{biane2005} and is proved by multidimensional versions of 
Pitman's transformation. It is also closely related to the analysis in \cite{borodin2008}. In the case that only rightward jumps in $Z_i$ are present, this corresponds to a construction of a process on a Gelfand-Tsetlin patten with pushing and blocking interactions \cite{warren_windridge}. 
The statement above can also be proved as a consequence of push-block dynamics by minor modifications of the proof of Theorem 2.1 in \cite{warren_windridge}
and we describe these modifications in Section \ref{Intertwining}. The construction of a corresponding process
on a symplectic Gelfand Tsetlin pattern in \cite{warren_windridge} leads to the following. 
\begin{lemma}[Theorem 2.3 of \cite{warren_windridge}]
\label{h_c_harmonic}
$h_C$ is harmonic for 
$(Z^{(v_n)}(t), \ldots, Z^{(v_1)})$ killed when it
leaves $W^n_{\leq 0}$.
\end{lemma}
This is a reflection through the origin of the result in \cite{warren_windridge} which considers
a process killed when it
leaves $W^n_{\geq 0}$.


\begin{proposition}[Corollary 7.7 of \cite{lecouvey}]
Suppose $0 < v_n < \ldots < v_1 < 1$.
\begin{enumerate}[(i)]
\item Let $T_A = \inf \{t \geq 0: (Z^{(v_n)}(t), \ldots, Z^{(v_1)}(t)) \notin W^n\}$.
Then for $x \in W^n$, we have $P_x(T_A = \infty) = \kappa_A h_A(x)$ where 
\begin{equation*}
\kappa_A = \prod_{i < j} (v_i - v_j) \prod_{j=1}^{n-1} v_{j}^{-(n-j)}.
\end{equation*}.
\item Let $T_C = \inf \{t \geq 0: (Z^{(v_n)}(t), \ldots, Z^{(v_1)}(t)) \notin W^n_{\leq 0}\}$. Then for $x \in W^n_{\leq 0}$, 
we have $P_x(T_C = \infty) = \kappa_C h_C(x)$
where 
\begin{equation*}
\kappa_C = \prod_{1 \leq i \leq j \leq n} (1 - v_i v_j ) \prod_{i < j} (v_i - v_j) \prod_{j=1}^{n-1} v_j^{-(n-j)} 
\end{equation*}
\end{enumerate}
\end{proposition}

The probability that a random walk remains within a Weyl chamber for all time is considered in a general setting in  \cite{lecouvey}. In our setting, we give a direct proof using Proposition \ref{two_sided_intertwining} and Lemma \ref{h_c_harmonic}.

\begin{proof}
Proposition \ref{two_sided_intertwining} and Lemma \ref{h_c_harmonic} show that $h_A$ and $h_C$ are harmonic functions for 
$(Z^{(v_n)}(t), \ldots, Z^{(v_1)})$ killed when it leaves $W^n$ and
$W^n_{\leq 0}$ respectively. 

We now check that $\kappa_A h_A$ and $\kappa_C h_C$ have the correct boundary behaviour. 
Let $\lvert x \rvert = \sum_{i=1}^d x_i$ for $x \in \mathbb{R}^d$ and define $\partial W^n = \{x \notin W^n : \exists x' \in W^n \text{ with } \lvert x - x' \rvert = 1\}$ and $\partial W^n_{\leq 0} = \{x \notin W^n_{\leq 0} : \exists x' \in W^n_{\leq 0} \text{ with } \lvert x - x' \rvert = 1\}$. Then we can observe from \eqref{schur_defn} that $S_x(v) = 0$ for all $x \in \partial W^n$ because two columns in the determinant in the numerator of \eqref{schur_defn} coincide if $x_i = x_{i+1} + 1$ for some $i = 1, \ldots, n-1$. In a similar manner, $h_C(x_1, \ldots, x_n) = \prod_{i=1}^n v_{n-i+1}^{-x_i} \text{Sp}_{(-x_n, \ldots, -x_1)}(v) = 0$ for all $x \in \partial W^n_{\leq 0}$ due to the above observation 
and that $h_C(x) = 0$ when $x_n = 1$. 

We now consider the behaviour at infinity. For $h_A$, it is easy to see from the Weyl character formula \eqref{schur_defn} that
\begin{equation*}
\kappa_A \lim_{x_i - x_{i+1} \rightarrow -\infty} \prod_{i=1}^n v_{n-i+1}^{-x_i} S_x(v) = \kappa_A \frac{\prod_{j=1}^{n-1} v_{j}^{n-j}}{\prod_{i < j} (v_i - v_j)}  = 1
\end{equation*}
where we use the limit above to mean $x_1, \ldots, x_n \rightarrow -\infty$ and $x_i - x_{i+1} \rightarrow -\infty$ for 
each $i = 1, \ldots, n-1$. 
For the symplectic Schur function we find \begin{equation*}
\lim_{x_{i} - x_{i+1} \rightarrow -\infty} \prod_{i=1}^n v_{n-i+1}^{-x_i} \text{Sp}_{(-x_n, \ldots, -x_1)}(v) = 
\frac{(-1)^{n} \prod_{j = 1}^n v_j^{-j}}{\text{det}(v_i^j - v_i^{-j})_{i, j = 1}^n}
\end{equation*}
and use Eq. 24.17 from \cite{Fulton_Harris} to give a more explicit expression for the limiting constant
\begin{equation}
\label{fulton_harris_eqn}
\text{det}(v_i^j - v_i^{-j})_{i, j = 1}^n = (-1)^{n} \prod_{i < j} (v_i - v_j) 
\prod_{1 \leq i \leq j \leq n}(1 - v_i v_j ) \prod_{j=1}^n v_j^{-n}.
\end{equation}
We conclude that 
 \begin{equation*}
\lim_{x_{i} - x_{i+1} \rightarrow -\infty} \kappa_C \prod_{i=1}^n v_{n-i+1}^{-x_i} \text{Sp}_{(-x_n, \ldots, -x_1)}(v) = 1.
\end{equation*}

In the case $0 < v_n < \ldots < v_1 < 1$ the process $(Z^{(v_n)}(t), \ldots, Z^{(v_1)}(t))$ almost surely has $Z_i^{(v_i)} \rightarrow -\infty$ for each $i = 1, \ldots, n$ and $Z_i^{(v_i)} - Z_{i+1}^{(v_{i+1})} \rightarrow -\infty$ for $i=1, \ldots, n-1$.
Therefore the above specifies the boundary behaviour of $\kappa_A h_A$ and $\kappa_C h_C$.

Suppose that $(h, T)$ either equals $(\kappa_A h_A, T_A)$ or $(\kappa_C h_C, T_C)$ and
let $Z_t^*$ denote $(Z^{(v_n)}(t), \ldots, Z^{(v_1)}(t))$ killed at the instant it leaves $W^n$ or $W^n_{\leq 0}$.
Then $(h(Z_t^*))_{t \geq 0}$ is a bounded martingale and
converges almost surely and in $L^1$ to a random variable $\mathcal{Y}$. From the boundary behaviour specified above, $\mathcal{Y}$ equals $1$ if $T = \infty$ 
and equals zero otherwise almost surely. Using this in the $L^1$ convergence shows that $h(x) = \lim_{t \rightarrow \infty} E_x(h(Z_t^*)) = P_x(T = \infty).$
\end{proof}

From this we can prove the second equality in law in Theorem \ref{top_particle_theorem} for a particular choice of rates. 
Suppose $0 < v_n < \ldots < v_1 < 1$ which ensures that all of the following events have strictly positive probabilities, 
and let $x \in W^n_{\leq 0}$. Then
\begin{IEEEeqnarray*}{rCl}
P_{(x_1, \ldots, x_n)}\left(\sup_{t \geq 0} Z_n^{\dagger} < 0\right) 
& = &
\frac{P_{(x_1, \ldots, x_n)}\left(T_C = \infty \right)}{P_{(x_1, \ldots, x_n)}(T_A = \infty)} \\
& = &  \frac{\kappa_C h_C(x)}{\kappa_A h_A(x)} \\
& = &  \prod_{1 \leq i \leq j \leq n} (1 - v_i v_j) \frac{ \text{Sp}_{(-x_n, \ldots, -x_1)}(v)}{\text{S}_x(v)}.
\end{IEEEeqnarray*} 
Let $(x_1, \ldots, x_n) \rightarrow (-\eta, \ldots, -\eta)$ and shift co-ordinates by $\eta$. Then  
\begin{equation}
\label{largest_particle_non_colliding_sup}
P_{0}\left(\sup_{t \geq 0} Z_n^{\dagger} \leq \eta \right) =  
P_{-\eta}\left(\sup_{t \geq 0} Z_n^{\dagger} \leq 0 \right)
 =  \prod_{1 \leq i \leq j \leq n} (1 - v_i v_j) \prod_{i = 1}^n v_{i}^{\eta}\text{Sp}_{\eta^{(n)}}(v)
\end{equation}
by using that $S_{(x_1, \ldots, x_n)}(v) \rightarrow \prod_{i=1}^n v_i^{-\eta}$ and the notation $\eta^{(n)} = (\eta, \ldots, \eta)$. 
We compare this to Corollary 4.2 of \cite{Bisi_Zygouras2} which in our notation states that
\begin{equation}
\label{lpp_formula}
P(G(1, 1) \leq \eta) =  \prod_{1 \leq i \leq j \leq n} (1 - v_i v_j) \prod_{i = 1}^n v_i^{\eta} 
\text{Sp}_{\eta^{(n)}}(v).
\end{equation}
Equation \eqref{largest_particle_non_colliding_sup} and \eqref{lpp_formula} prove the second equality in law in Theorem \ref{top_particle_theorem} for $0 < v_n < \ldots < v_1 < 1$. 
This can be extended to all distinct rates with $v_i < 1$ for each $i = 1, \ldots, n$ by observing that the law of the process $(Z_n^{\dagger}(t))_{t \geq 0}$ 
is invariant under permutations of the $v_i$. In particular, this holds for $\sup_{t \geq 0} Z_n^{\dagger}(t)$ and 
also holds for $G(1, 1)$ from \eqref{lpp_formula}.

\subsection{Time reversal}

We now prove that PushASEP with a wall started from $(0, \ldots, 0)$ has an interpretation as semi-discrete last passage percolation times in a 
environment constructed from $2n$ Poisson point processes.
In particular, 
\begin{equation}
\label{semi_discrete_lpp}
(Y_k(t))_{k=1}^n = \left(\sup_{0 \leq t_0 \leq t_1 \leq \ldots \leq t_k = t} \sum_{i=1}^k (Z_i^{(v_i)}(t_i) - Z_i^{(v_i)}(t_{i-1}))\right)_{k=1}^n
\end{equation}
where the $Z_i^{(v_i)}$ are a difference of two Poisson point processes.
In the proof of \eqref{semi_discrete_lpp} we will denote 
the right hand side of \eqref{semi_discrete_lpp} by $(U_k(t))_{k = 1}^n$. We check that the evolution of this process is PushASEP with a wall. When $n = 1$, $U_1(t) = \sup_{0 \leq t_0 \leq t} (Z_1^{(v_1)}(t) - Z_1^{(v_1)}(t_0))$ 
and this evolves as PushASEP with a wall with one particle started from zero.  
For the inductive step we note that adding in the $n$-th particle to $(U_k(t))_{k = 1}^n$ does not affect the 
evolution of the first $(n-1)$ particles. Therefore we only need to consider the $n$-th particle given by
\begin{equation}
\label{poisson_skorokhod}
U_n(t) = \sup_{0 \leq s \leq t} (Z_n^{(v_n)}(t) - Z_n^{(v_n)}(s) + Y_{n-1}(s))  
\end{equation}
where $Y_{n-1}$ is the $(n-1)$-th particle in PushASEP with a wall. 
If $U_{n} > Y_{n-1}$ then the suprema in \eqref{poisson_skorokhod} is attained with a choice $s < t$ 
and $U_n$ jumps right or left whenever $Z_n^{(v_n)}$ does. If $U_n = Y_{n-1}$ then at least one of the 
(possibly non-unique) maximisers of the supremum in \eqref{poisson_skorokhod} involves $s = t$. 
This means that if $Z_n^{(v_n)}$  jumps to the right then $U_n$ jumps to the right; 
if $Y_{n-1}$ jumps to the right then $U_n$ jumps to the right (this is is the \emph{pushing} interaction); and if 
$Z_n^{(v_n)}$ jumps to the left then $U_n$ is unchanged (this is the \emph{blocking} interaction).
Therefore $U_n$ defined by \eqref{poisson_skorokhod} follows the dynamics of the $n$-th particle in PushASEP
with a wall started from the origin.
Therefore \eqref{semi_discrete_lpp} follows inductively.

%


Equation \ref{semi_discrete_lpp} has a similar form to Proposition \ref{two_sided_intertwining} and 
this along with time reversal establishes the following connection, see \cite{five_author, FW} for a similar 
argument in a 
Brownian context.
\begin{proposition}
\label{push_asep_non_colliding}
Let $Y_n^*$ be distributed as the top particle in PushASEP with a wall in its invariant measure and 
$Z_n^{\dagger}$ be the top particle in the ordered random walk with $Q$-matrix given by \eqref{non_colliding_defn} and started from the origin. Then
\begin{equation*}
Y_n^* \stackrel{d}{=} \sup_{t \geq 0} Z_n^{\dagger}(t).
\end{equation*}
\end{proposition}

\begin{proof}
For any fixed $t$, we let $t - u_i = t_{k-i}$ and use time reversal of continuous-time random walks 
$(Z_i^{(v_i)}(t) - Z_i^{(v_i)}(t - s))_{s \geq 0} \stackrel{d}{=} (Z_{n-i+1}^{(v_i)}(s))_{s \geq 0}$ to establish that 
\begin{IEEEeqnarray}{rCl}
\label{time_reversal}
(Y_k(t))_{k=1}^n & = &  \left(\sup_{0 \leq t_0 \leq \ldots \leq t_k = t} \sum_{i=1}^k (Z_i^{(v_i)}(t_i) - Z_i^{(v_i)}(t_{i-1}))\right)_{k=1}^n \IEEEnonumber\\
& = &  \left(\sup_{0 = u_0 \leq \ldots \leq u_k \leq t} \sum_{i=1}^k (Z_i^{(v_i)}(t- u_{k-i}) - Z_i^{(v_i)}(t - u_{k-i+1})) 
\right)_{k=1}^n \IEEEnonumber\\
& \stackrel{d}{=} & \left(\sup_{0 = u_0 \leq \ldots \leq u_k \leq t} \sum_{i=1}^k (Z_{n-i+1}^{(v_i)}(u_{k-i+1}) - Z_{n-i+1}^{(v_i)}(u_{k-i})) \right)_{k=1}^n
\end{IEEEeqnarray}
The equality in law of the largest co-ordinates, relabelling the sum from $i$ to $n-i+1$ and comparing with Proposition \ref{two_sided_intertwining} part (ii) shows that $Y_n(t) \stackrel{d}{=} \sup_{0 \leq s \leq t} Z_n^{\dagger}(s)$. In particular, letting $t \rightarrow \infty$ completes the proof.
\end{proof}

\begin{lemma}
\label{continuity_invariant_measure}
The distribution of $(Y_1^*, \ldots, Y_n^*)$ is continuous in $(v_1, \ldots, v_n)$ on the set $(0, 1)^n$.
\end{lemma}

\begin{proof}
We will use the representation for $(Y_1^*, \ldots, Y_n^*)$ obtained by relabelling the sum $i$ to $k-i+1$ in \eqref{time_reversal} 
 and letting
$t \rightarrow \infty$,
\begin{equation*}
 (Y_k^*)_{k=1}^n \stackrel{d}{=} \left(\sup_{0 = t_0 \leq \ldots \leq t_k < \infty} \sum_{i=1}^k (Z_{n-k+i}^{(v_{k-i+1})}(t_{k-i+1}) - Z_{n-k+i}^{(v_{k-i+1})}(t_{k-i})) \right)_{k=1}^n.
\end{equation*}
We fix $\epsilon > 0$ and construct realisations of $Z_i^{(v_{n-i+1})}$ for all $\epsilon < v_i < 1 - \epsilon$ on the same probability space. 
To achieve this we define $2n$ independent marked Poisson point process $R_1, \ldots, R_n$ and $L_1, \ldots, L_n$ on
$\mathbb{R}_{\geq 0} \times [0, 1]$ which will dictate the rightwards and leftward jumps respectively of 
$Z_i^{(v_{n-i+1})}$. 
For each $i = 1, \ldots, n$, the marked Poisson point process $R_i$ and $L_i$ consist of points $(t_k, w_k)_{k \geq 1}$ and $(\bar{t}_k, \bar{w}_k)_{k \geq 1}$ where $(t_k)_{k \geq 1}$ and $(\bar{t}_k)_{k \geq 1}$ are the points
of a Poisson point process of rate $1$ and $1/\epsilon$ respectively on $\mathbb{R}_{\geq 0}$. The $w_i$ and 
$\bar{w}_i$ are uniform random variables on the interval $[0, 1]$ which are independent of each other and $(t_i)_{i \geq 1}, (\bar{t}_i)_{i \geq 1}$. 

We define $R_i^{(v)}$ to be the subset of  $(t_i, w_i)_{i \geq 1}$ with $w_i > 1-v$ and $L_i^{(1/v)}$ to be the subset of $(\bar{t}_i, \bar{w}_i)_{i \geq 1}$ with $\bar{w}_i > 1 - \epsilon/v$.
The projection onto the first co-ordinate of $R_i^{(v_i)}$ and $L_i^{(v_i)}$ give independent Poisson point process of rate $v_i$ and $1/v_i$ respectively which define coupled realisations of $Z_i^{(v_{n-i+1})}$ for any choice of 
$\epsilon < v < 1-\epsilon$. 

Almost surely, the suprema on the right hand side of Proposition \ref{semi_discrete_lpp} part (ii) all stablise (after some random time uniform over $(v_1, \ldots, v_n) \in (\epsilon, 1- \epsilon)^n$). For any realisation of the marked Poisson point processes, the right hand side of Proposition \ref{semi_discrete_lpp} part (ii)  is continuous 
in $(v_1, \ldots, v_n)$ except at $(1-w_i)_{i \geq 1}$ and $(\epsilon/(1-w_i))_{i \geq 1}$. 
Therefore the distribution of the right hand side of part (ii) of Proposition \ref{semi_discrete_lpp} is continuous 
in $(v_1, \ldots, v_n)$ on the set $(\epsilon, 1-\epsilon)^n$, and hence so is the distribution of $(Y_1^*, \ldots, Y_n^*)$. As $\epsilon$ is arbitrary this completes the proof.
\end{proof}

\begin{proof}[Proof of Theorem \ref{top_particle_theorem}]
Proposition \ref{push_asep_non_colliding} is the first equality in law.
At the end of Section \ref{suprema_non_colliding} we proved the second equality for distinct $0 < v_1, \ldots, v_n < 1$. Lemma \ref{continuity_invariant_measure} allows us to remove the constraint that the $v_i$ are distinct.
\end{proof}

\section{Push ASEP with a wall}
\label{push_asep}

\subsection{Transition probabilities}
We give a more explicit definition of Push-ASEP with a wall at the origin as a continuous-time Markov chain $(Y(t))_{t \geq 0} = (Y_1(t), \ldots,  Y_n(t))_{t \geq 0}$ taking values in $W^n_{\geq 0} = \{(y_1, \ldots, y_n) : y_i \in \mathbb{Z} \text { and } 0 \leq y_1 \leq y_2 \ldots \leq y_n\}$. We use $e_i$ to denote the vector taking value $1$ in position $i$ and zero otherwise.
The transition rates of $Y$ are defined for
$y, y + e_i + \ldots + e_j \in W^n_{\geq 0}$ and $i \leq j$  by
\begin{equation}
\label{defn_PushASEP1}
q(y, y+e_i + e_{i+1} + \ldots + e_j) = v_i 1_{\{ y_i =
y_{i+1} = \ldots = y_j < y_{j+1}\}} 
\end{equation}
with the notation $y_{n+1} = \infty$
and for $y \in W^n_{\geq 0}$ by
\begin{equation}
\label{defn_PushASEP2}
q(y, y-e_i) = v_i^{-1} 1_{\{ y - e_i \in W^n_{\geq 0}\}}.
\end{equation} 
All other transition rates equal zero.
We note that in \cite{borodin2008} the particles were strictly ordered, whereas it is convenient for us to 
consider a weakly ordered system; these systems can be related by a co-ordinate change 
$x_j \rightarrow x_j + j -1$.  

To describe the transition probabilities we first introduce the operators acting on functions $f: \mathbb{Z} \rightarrow \mathbb{R}$ with $v > 0$,
\begin{equation*}
D^{(v)}f(u) = f(u) - vf(u-1), \qquad J^{(v)}f(u) = \sum_{j = u}^{\infty} v^{u - j} f(j),
\end{equation*}
where we will always apply $J^{(v)}$ to functions with superexponential decay at infinity.
We use $D^{(v_1, \ldots, v_n)} = D^{(v_1)} \ldots D^{(v_n)}$ and 
$J^{(v_1, \ldots, v_n)} = J^{(v_1)} \ldots J^{(v_n)}$ as notation for concatenated operators
and $D_u^{(v)}, J_u^{(v)}$ to specify a variable $u$ on which the operators act.

We recall Siegmund duality for birth-death processes, see for example \cite{clifford_sudbury, cox_rosler}. 
Let $(X_t)_{t \geq 0}$ denote a birth-death process on the state space $\mathbb{Z}_{\geq 0}$ with transition rates:
\begin{equation*}
i \rightarrow i+1 \text{ rate } \lambda_i \text{ for } i \geq 0, \qquad i \rightarrow i - 1 \text{ rate } \mu_i \text{ for } i \geq 1.
\end{equation*}
Let $(X^*_t)_{t \geq 0}$ denote a birth-death process on the state space $\mathbb{Z}_{\geq -1}$ 
with transition rates
\begin{equation*}
i \rightarrow i+1 \text{ rate } \mu_{i+1} \text{ for } i \geq 0, \qquad i \rightarrow i-1 \text{ rate } \lambda_i
\text{ for } i \geq 0.
\end{equation*}
The process $X$ has a reflecting boundary at zero while $X^*$ is absorbed at $-1$.
Under suitable conditions on the rates, see \cite{cox_rosler}, which hold in the case
of interest to us: $\lambda_i = v_1$ for $i \geq 0$
and $\mu_i = v^{-1}_1$ for $i \geq 1$, Siegmund duality states that 
\begin{equation*}
P_{x}(X_t \leq y) = P_{y}(X_t^* \geq x) \quad \text{ for all } t \geq 0, \quad x, y \in \mathbb{Z}_{\geq 0}.
\end{equation*}
We can find the transition probabilities for $X^*$ by solving the Kolmogorov forward equation.
We define for any $t \geq 0$ and $x, y \in \mathbb{Z}$,
\begin{equation*}
\psi_t(x, y) = \frac{1}{2\pi i} \oint_{\Gamma_0} \frac{dz}{z} (z^{y-x} - z^{x + y + 2} ) e^{t(z + 1/z)}
\end{equation*}
where $\Gamma_0$ denotes the unit circle oriented anticlockwise. The transition probabilities of $X^*$ are given for $x \in \mathbb{Z}_{\geq 0}$ and $t \geq 0$ by $P_{x}(X^*_t = y)
= v_1^{y - x} e^{-t(v_1 + 1/v_1)} \psi_t(x, y)$ for 
$y \in \mathbb{Z}_{\geq 0}$ and $P_x(X^*_t = -1) = 1 - \sum_{y \geq 0} v_1^{y - x} e^{-t(v_1+1/v_1)} \psi_t(x, y)$. 
 
By using Siegmund duality, the transition probabilities of PushASEP with a wall with a single particle (which is an M/M/1 queue)
are given by 
\begin{equation*}
v_1^{y - x} e^{-t(v _1+ 1/v_1)} D_{y}^{(1/v_1)} J_{x}^{(v_1)}  \psi_t(x, y) \quad \text{ for all } t \geq 0, \quad x, y \in \mathbb{Z}_{\geq 0}.
\end{equation*}
The purpose of the above is that this now provides a form which is convenient to generalise
to $n$ particles. 
Define for all $t \geq 0$ and  $x, y \in \mathbb{Z}^n$, 
\begin{equation*}
r_t(x, y) = \prod_{k=1}^n v_k^{y_k - x_k} e^{-t(v_k + 1/v_k)}\text{det}(F_{ij}(t; x_i+i-1, y_j+j-1))_{i, j = 1}^n
 \end{equation*}
 with 
 \begin{equation*}
 F_{ij}(t; x_i + i - 1, y_j + j - 1) = D_{y_j}^{(1/v_1 \ldots 1/v_j)} J_{x_i}^{(v_1 \ldots v_i)} \psi_t(x_i + i - 1, y_j + j - 1). 
 \end{equation*}

\begin{proposition}
The transition probabilities of $(Y_1(t), \ldots, Y_n(t))_{t \geq 0}$ are given by $r_t(x, y)$ for $x, y \in W^n_{\geq 0}$.
\end{proposition}

The transition probabilities for PushASEP in the absence of a wall were found in \cite{borodin2008} and related
examples have been found in \cite{schutz, warren}. Our proof follows the ideas in \cite{borodin2008}.

\begin{proof}
Observe that for all $u, w \in \mathbb{Z}$, 
\begin{equation*}
\frac{d\psi}{dt} = \psi_t(u, w+1) + \psi_t(u, w-1)
\end{equation*}
and therefore for all $x, y \in W^n$,
\begin{equation}
\label{push_ASEP_free_eqn}
\frac{dr}{dt} = \sum_{k=1}^n v_k^{-1} r_t(x, y + e_k) + v_k r_t(x, y-e_k) -(v_k + v_k^{-1})r_t(x, y). 
\end{equation}
We note that $y \pm e_k$ may be outside of the set $W^n_{\geq 0}$ but that $r$ has been defined for 
all $x, y \in \mathbb{Z}^n$. The proof will involve showing that the terms involving $y \pm e_k \notin W^n_{\geq 0}$ in  \eqref{push_ASEP_free_eqn} can be replaced, using identities for $r$, by
terms corresponding to the desired pushing and blocking interactions. 

An important role is played by the identity, that if $y_j = y_{j+1}$ then 
\begin{equation}
\label{push_ASEP_identity}
v_j^{-1} r_t(x, y + e_j) = v_{j+1}^{-1} r_t(x, y). 
\end{equation}
This can be proved by showing that the difference of the two sides is equal to 
\begin{equation*}
\prod_{k = 1}^n v_k^{y_k - x_k} e^{-t(v_k + 1/v_k)} \text{det}(A_{ij})_{i, j = 1}^n
\end{equation*}
where the relevant columns of $A$ are the $k$-th and $(k+1)$-th which have entries for each $i = 1, \ldots, n$
given by
\begin{equation*}
A_{ik} = D_{y_{k+1}}^{(1/v_{k+1})} F_{ik}(x_i + i - 1, y_{k+1}+k), \qquad
A_{ik+1} = F_{i k+1}(x_i + i - 1, y_{k+1} + k).
\end{equation*}
These two columns are equal which proves \eqref{push_ASEP_identity}.

We first consider the terms in \eqref{push_ASEP_free_eqn} with $y - e_k \notin W^n_{\geq 0}$ 
and $y_k > 0$
which corresponds to right jumps with a pushing interaction. 
Denote by $m(k)$ the minimal index such that $y_{m(k)} = y_{m(k+1)} = \ldots = y_k$. 
Then by iteratively applying the identity \eqref{push_ASEP_identity} we obtain
\begin{IEEEeqnarray*}{rCl}
v_k r_t(x, y-e_k) & = & v_{k-1} r_t(x, y - e_{k-1} - 
e_{k}) = \ldots = v_{m(k)} r_t(x, y - e_{m(k)} - \ldots
- e_{k-1} - e_{k}). 
\end{IEEEeqnarray*}
This shows that,
\begin{multline}
\label{push_ASEP_push}
\sum_{k=1}^n v_k r_t(x, y - e_k) 1_{\{y_{m(k-1)} < y_{m(k)} = \ldots = y_k\}} - v_k r_t(x, y)
\\ 
= \sum_{k=1}^n v_{m(k)} r_t(x, y- e_{m(k)} - \ldots
- e_{k-1} - e_{k}) 1_{\{y_{m(k-1)} < y_{m(k)} = \ldots = y_k\}} - v_k r_t(x, y)
\end{multline}
where $y_{0} := 0$. We note that 
$y-e_{m(k)} - \ldots
-e_{k-1} - e_{k} \in W^n_{\geq 0}$ whenever $y_{m(k-1)} < y_{m(k)} = \ldots = y_k$ holds.

We next consider the terms in \eqref{push_ASEP_free_eqn} with $y + e_k \notin W^n_{\geq 0}$ which 
will correspond to blocking interactions.
This means that $y_k = y_{k+1}$ and using \eqref{push_ASEP_identity} shows that 
\begin{equation*}
\sum_{k=1}^{n-1} v_k^{-1} 1_{\{y_k = y_{k+1}\}} r_t(x, y + e_k)
= \sum_{k=2}^{n} v_k^{-1} 1_{\{y_{k-1} = y_{k}\}} r_t(x, y).
\end{equation*}
Therefore
\begin{multline}
\label{push_ASEP_block}
 \sum_{k=1}^n v_k^{-1} r_t(x, y + e_k) -v_k^{-1}r_t(x, y)
   =   -v_1^{-1} r_t(x, y) - \sum_{k=2}^n v_k^{-1}(1 - 1_{\{y_{k-1} = y_k\}} )
 r_t(x, y) \\
 + v_n^{-1} r_t(x, y+ e_n) + \sum_{k=1}^{n-1} v_k^{-1} (1- 1_{\{y_k = y_{k+1}\}})r_t(x, y+ e_k).
\end{multline} 
We note that $y+ e_k \in W^n_{\geq 0}$ whenever $(1- 1_{\{y_k = y_{k+1}\}}) \neq 0$.

The final terms we need to consider in \eqref{push_ASEP_free_eqn} are those with $y - e_k \notin W^n_{\geq 0}$ 
and $y_k = 0$ which correspond to left jumps which are suppressed by the wall.
If $y_1 = \ldots = y_k = 0$ for some $k > 1$, then 
\begin{equation*}
r_t(x, y-e_k) =  v_k^{-1} \prod_{j=1}^n v_j^{y_j - x_j} e^{-t(v_j + 1/v_j)} \text{det}(B_{ij})_{i, j = 1}^n 
\end{equation*}
for a matrix $B$ where the relevant entries of $B$ are the columns indexed by $1, \ldots, k$.
The first column has entries $B_{i1} = J_{x_i}^{(v_1 \ldots v_i)} 
(\psi(x_i - i+1, 0) - v_1^{-1} \psi(x_i - i+1, -1))$ which simplifies to $B_{i1} = J_{x_i}^{(v_1 \ldots v_i)} 
\psi(x_i - i+1, 0)$ by the fact that $\psi(\cdot, -1) = 0$. 
The columns indexed by $j = 2, \ldots, k-1$ can be simplified to
 $B_{ij} = J_{x_i}^{(v_1 \ldots v_i)} \psi(x_i - i+1, j-1)$ by using $\psi(\cdot, -1) = 0$ and column operations. 
Using this argument for the $k$-th column and that we consider the vector $y - e_k$, 
we observe that the $k$-th column is a linear combination of columns $1, \ldots, k-1$ and hence 
$r_t(x, y - e_k) = 0$ if $y_k = 0$ for any $k > 1$.

The remaining case is when $0 = y_1 < y_2$ and we show that
\begin{equation*}
v_1  r_t(x, y-e_1) - v_1^{-1} r_t(x, y) = 0.
\end{equation*}
This follows from multilinearity of the determinants involved in the definition of $r$ and using $\psi(\cdot, - 1) = 0$,
\begin{equation*}
D^{(1/v_1)} \psi(x_i, -1) -   v_1^{-1}  D^{(1/v_1)} \psi_t(x_i, 0) = 
-v_1^{-1}\psi(x_i, -2) - v_1^{-1} \psi_t(x_i, 0)  = 0. 
\end{equation*}
Therefore 
\begin{equation}
\label{push_ASEP_wall}
\sum_{k=1}^n v_k r_t(x, y-e_k)  1_{\{0 = y_1 = \ldots = y_k\}} = v_1^{-1} 1_{\{y_1 = 0\}}  r_t(x, y).
\end{equation}

We combine \eqref{push_ASEP_free_eqn}, \eqref{push_ASEP_push}, \eqref{push_ASEP_block}
and \eqref{push_ASEP_wall} to obtain that 
\begin{IEEEeqnarray}{rCl}
\label{forward_eq1}
\frac{dr}{dt} & = & \sum_{k=1}^n v_{m(k)} r_t(x, y-e_{m(k)} - \ldots
-e_{k-1} - e_{k})1_{\{y_{m(k-1)} < y_{m(k)} = \ldots = y_k\}} - \sum_{k=1}^n v_k r_t(x, y) \IEEEnonumber
\\ & &
 -v_1^{-1}(1- 1_{\{y_1 = 0\}}) r_t(x, y) 
  - \sum_{k=2}^n v_k^{-1}(1 - 1_{\{y_{k-1} = y_k\}} )
 r_t(x, y) \IEEEnonumber \\
 & &
 + v_n^{-1} r_t(x, y+ e_n) + \sum_{k=1}^{n-1} v_k^{-1} (1- 1_{\{y_k = y_{k+1}\}})r_t(x, y+ e_k).
\end{IEEEeqnarray}

We now consider the initial condition. 
\begin{equation}
\label{r_0_eqn}
r_0(x, y) = \prod_{k=1}^n v_k^{y_k - x_k} \text{det}(F_{ij}(0; x_i+i-1, y_j+j-1))_{i, j = 1}^n
 \end{equation}
 where 
  \begin{equation*}
 F_{ij}(0; x_i+i-1, y_j+j-1) = D_{y_j}^{(1/v_1 \ldots 1/v_j)} J_{x_i}^{(v_1 \ldots v_i)} \psi_0(x_i+i-1, y_j+j-1)
 \end{equation*}
 and $\psi_0(u, w) = 1_{\{w - u = 0\}}$ for $u, w \geq 0$ depends only on the difference $w-u$ 
 and we will view this as a function of $w-u$.
For any function $f : \mathbb{Z} \rightarrow \mathbb{R}$ and $u, w \in \mathbb{Z}$, $r > 0$
\begin{equation}
\label{JD_translation_invariance}
D_w^{(1/r)} J_u^{(r)} f(w - u) = D_w^{(1/r)} \left(\sum_{k = u}^{\infty} r^{u-k} f(w - k)\right) = f(w - u).
\end{equation}
Therefore the top-left entry in the matrix defining $r_0$ equals $1_{\{y_1 = x_1\}}$. 
Suppose $y_1 > x_1$ and observe that if a function $g$ has $g(u) = 0$ for $u > 0$, 
then for any $j=1, \ldots, n$ we have 
$D^{(1/v_2 \ldots 1/v_j)}_u g(u+j-1) = 0$ for $u > 0$. This shows that when $y_1 > x_1$ the top row of the matrix defining $r_0$ equals zero. 
In a similar manner, when $y_1 < x_1$ the first column in the matrix defining $r_0$ is zero. 
Therefore 
\begin{equation}
\label{r_0_inductive}
r_0(x, y) = 1_{\{x_1 = y_1\}} \prod_{k=2}^n v_k^{y_k - x_k} \text{det}(F_{ij}(0; x_i+i-1, y_j+j-1))_{i, j = 2}^n
 \end{equation}
and using \eqref{JD_translation_invariance} the entries of the matrix in \eqref{r_0_inductive} 
have the same form as the entries of the matrix in \eqref{r_0_eqn} but with $n-1$ particles. Continuing inductively, 
\begin{equation}
\label{forward_eq2}
r_0(x, y) = \prod_{k = 1}^n 1_{\{x_k = y_k\}}. 
\end{equation}

Therefore $r_t(x, y)$ satisfies the Kolmogorov forward equations \eqref{forward_eq1} and \eqref{forward_eq2} corresponding to 
the process $(Y(t))_{t \geq 0}$. These equations have a unique solution given by the transition probabilities of 
$(Y(t))_{t \geq 0}$ because the process does 
not explode. 
\end{proof}

\begin{lemma}
\label{ibp_lemma}
Let $(f_i)_{i =1}^n$ and $(g_j)_{j = 1}^n$ be functions $\mathbb{Z}_{\geq - 1} \rightarrow \mathbb{R}$ such that $g_j$ decays superexponentially while $f_i$ grows at most exponentially at infinity. 
\begin{enumerate}[(i)]
\item Suppose further that $f_i(-1) = 0$ for each $i = 1, \ldots, n$, 
\begin{multline*}
\sum_{x \in W^n_{\geq 0}} \text{det}(D^{(1/v_1 \ldots 1/v_j)} f_i(x_j + j -1 ))_{i, j = 1}^n
 \text{det}(J^{(v_1 \ldots v_i)} g_j(x_i + i -1 ))_{i, j = 1}^n \\ = 
 \text{det}\left(\sum_{u \geq 0} f_i(u) g_j(u)\right)_{i, j = 1}^n.
\end{multline*}
\item With no extra conditions and the notation $D^{\emptyset} = \text{Id}$,
\begin{multline*}
\sum_{x \in W^n_{\geq 0}} \text{det}(D^{(1/v_2 \ldots 1/v_j)} f_i(x_j + j -1 ))_{i, j = 1}^n
 \text{det}(J^{(v_2 \ldots v_i)} g_j(x_i + i -1 ))_{i, j = 1}^n \\ = 
 \text{det}\left(\sum_{u \geq 0} f_i(u) g_j(u)\right)_{i, j = 1}^n.
\end{multline*}
\end{enumerate}
\end{lemma}

\begin{proof}
The proof is similar to Lemma 2 in \cite{FW} and so we give 
a description of the proof and 
refer to \cite{FW} which carries out some of the steps more explicitly. We prove (i) first and (ii) is almost identical.

We first observe that 
\begin{equation}
\label{ibp_identity}
\sum_{u = a}^b (D^{(1/v)}f)(u+1) (J^{(v)} g)(u+1) = \sum_{u = a}^b f(u) g(u) + f(b+1) (J^{(v)}g)(b+1) - f(a) (J^{(v)}g)(a).
\end{equation}
We apply \eqref{ibp_identity} repeatedly to show that 
\begin{multline}
\label{ibp}
\sum_{x \in W^n_{\geq 0}} \text{det}(D^{(1/v_1 \ldots 1/v_j)} f_i(x_j + j -1 ))_{i, j = 1}^n
 \text{det}(J^{(v_1 \ldots v_i)} g_j(x_i + i -1 ))_{i, j = 1}^n \\
 = \sum_{x \in W^n_{\geq 0}} \text{det}(f_i(x_j-1 ))_{i, j = 1}^n
 \text{det}( g_j(x_i-1 ))_{i, j = 1}^n.
\end{multline}
The general procedure is to use a Laplace expansion of the determinants on the left hand side, apply \eqref{ibp_identity} with a particular choice of variable and parameter, and then reconstruct the 
result as a sum of three determinants. A key property is that all of the boundary terms in \eqref{ibp_identity}
will end up contributing zero. 

The first application of this procedure is with the parameter $v_n$, variable $x_n$ and summing $x_n$  
from $x_{n-1}$ to infinity.
This shows that the left hand side of \eqref{ibp} equals a sum of three terms which all take the form:
\begin{equation*}
\sum_{\Sigma} \text{det}(A_{ij})_{i, j = 1}^n \text{det}(B_{ij})_{i, j = 1}^n.
\end{equation*}
In the first term,  $\Sigma = \{(x_1, \ldots, x_n) \in W^n_{\geq 0}\}$. The $A_{ij}$ are given by the entries of the first matrix on the left hand side of \eqref{ibp} except with the 
application of $D^{(1/v_n)}$ in the $n$-th column removed and the argument $x_n + n - 1$ replaced by 
$x_n + n - 2$. The $B_{ij}$ are given by the entries of the second matrix on the left hand side of \eqref{ibp} except with the 
application of $J^{(1/v_n)}$ in the $n$-th row removed and the argument $x_n + n - 1$ replaced by 
$x_n + n - 2$. 
There are two boundary terms which have $\Sigma = \{(x_1, \ldots, x_{n-1}) \in W_{n-1}^+\}$ and are
 evaluated at $x_n = x_{n-1}$ and $x_{n} = \infty$. These terms are both zero: 
 when evaluated at $x_n = x_{n-1}$ 
two columns in $A_{ij}$ are equal, and the boundary term at infinity vanishes due to the growth and decay conditions 
imposed on $f$ and $g$.

We continue this process of using \eqref{ibp_identity} with the following orders of parameters and variables:
$(x_n, v_n), (x_{n-1}, v_{n-1}), \ldots, (x_1, v_1), (x_{n-1}, v_n), (x_{n-2}, v_{n-1}), \ldots, (x_2, v_1), 
(x_n, v_1)$. For each $j=2, \ldots, n-1$ the sum in \eqref{ibp_identity} when applied to the $x_j$ variable is from $x_{j-1}$ to $x_{j+1}$ 
and all boundary terms are zero. In the generic case, the boundary term corresponding to the upper limit of summation 
is evaluated at $x_j = x_{j+1} + 1$ and is zero because two rows in the determinant of $B_{ij}$ are equal. When \eqref{ibp_identity} is applied to the $x_1$ variable
the sum is from $0$ to $x_2$ and the boundary term at zero vanishes by the condition that $f_i(-1) = 0$ 
for each $i = 1, \ldots, n$. 

This proves \eqref{ibp} and we apply the Cauchy-Binet (or Andr\'eief) identity to the right hand side of 
\eqref{ibp}
to complete the proof of part (i).

Part (ii) is identical except that we do not apply \eqref{ibp_identity} to the $x_1$ variable. Thus the condition
$f_i(-1) = 0$ for each $i=1, \ldots, n$ can be omitted.
\end{proof}

\subsection{Invariant measure}
\begin{proposition}
\label{Invariant_measure}
Let $(Y_1^*, \ldots, Y_n^*)$ be distributed according to the 
invariant measure of PushASEP with a wall and suppose that 
the rates $0 < v_1, \ldots, v_n < 1$ are distinct. Then the probability mass function of $(Y_1^*, \ldots, Y_n^*)$ 
is given by
\begin{equation*}
 \pi(x_1, \ldots, x_n) = c_n\prod_{k=1}^n v_k^{x_k} \text{det}( D^{(1/v_1 \ldots 1/v_j)} \phi_i(x_j + j - 1))_{i, j = 1}^n
\end{equation*}
where $\phi_i(x) =  v_i^{-(x+1)} - v_i^{x+1}$ and 
$c_n = \prod_{1 \leq i < j \leq n}  \frac{1}{(v_i - v_j)} \prod_{j = 1}^n v_j^n$.
\end{proposition}
We note that the Markov chain is irreducible, does not explode and 
the invariant measure is unique when normalised.

\begin{proof}
We use Lemma \ref{ibp_lemma}, noting that $\phi_i(-1) = 0$ for each $i$ and that the conditions at infinity are satisfied, to find
\begin{equation*}
\sum_{x \in W^n_{\geq 0}} \pi(x) r_t(x, y) 
= c_n \prod_{k=1}^n v_k^{y_k} e^{-t(v_k + 1/v_k)} \text{det}\left(D_{y_j}^{(1/v_1 \ldots 1/v_j)} \sum_{u \geq 0} 
\phi_i(u) \psi_t(u, y_j + j -1 )\right)_{i, j = 1}^n.
\end{equation*}
We recall that $\psi$ is related to the transition probabilities
of a process $(X^*_t)_{t \geq 0}$ defined through two independent Poisson point processes $N_t^{(1)}$ and $N_t^{(2)}$ of rate 
$1$ as $X_t^* = N_t^{(1)} - N_t^{(2)}$ for all $0 \leq t \leq \tau_{-1}$ and $X_t^* = -1$ for all $t > \tau_{-1}$, where 
$\tau_{-1} = \inf\{t\geq 0: N_t^{(2)} = N_t^{(1)}+1\}$. 
The transition probabilities of $X^*$ are given for $\xi \geq 0$ by $P_{\xi}(X_t^* = \eta) = e^{-2t} \psi_t(\xi, \eta)$ for $\eta \geq 0$ and 
$P_{\xi}(X_t^* = -1) = 1 - \sum_{\eta \geq 0} e^{-2t} \psi_t(\xi, \eta)$. On the other hand, 
$(v^{-(X^*_t + 1)} - v^{X_t^*+1}) e^{-(v + 1/v)t + 2t}$ is a martingale for $X^*$ for any $v > 0$.  
In particular, 
\begin{equation*}
\sum_{u \geq 0}  \psi_t(y_j + j -1, u)\phi_i(u) e^{-t(v_i + 1/v_i)} = \phi_i(y_j + j - 1).
\end{equation*}
Using this and the fact that $\psi$ is symmetric in the right hand side of the first displayed equation in this proof
shows that 
\begin{equation*}
\sum_{x \in W^n_{\geq 0}} \pi(x) r_t(x, y)  = c_n \prod_{k=1}^n v_k^{y_k} \text{det}(D^{(1/v_1 \ldots 1/v_j)} \phi_i(y_j + j - 1))_{i, j = 1}^n = \pi(y).
\end{equation*}
We defer the proof that $\pi$ is positive and the identification of the normalisation constant. These two 
properties will follow by identifying $\pi$ as 
the probability mass function for a vector of last passage percolations times in the proof of Theorem \ref{main_theorem}.
\end{proof}

\section{Point-to-line last passage percolation}
\label{point_to_line}

Point-to-line last passage percolation can be interpreted as an interacting particle system, 
where at each time step a new particle is added at the origin and particles interact by pushing particles to the right of them. 
We define a discrete-time Markov chain
denoted $(\mathbf{G}^{pl}(k))_{1 \leq k \leq n}$ where $\mathbf{G}^{pl}(k) = (G_1^{pl}(k), \ldots, G_k^{pl}(k))$. 
The particles are updated between time $k-1$ and time $k$ 
by sequentially defining $G_1^{pl}(k), \ldots, G_n^{pl}(k)$
starting with
$G_1^{pl}(k) = g_{n-k+1, k}$ and then applying the update rule 
\begin{equation}
\label{update_rule_lpp}
G_j^{pl}(k) = \max(G_j^{pl}(k-1), G_{j-1}^{pl}(k)) + g_{n-k+1, k-j+1}, \qquad \text{ for } 2 \leq j \leq k
\end{equation}
where $(g_{jk})_{j, k \geq 1, j+k\leq n+1}$ are an independent collection of geometrically distributed random variables with parameters $1 - v_j v_{n-k+1}$ and $0 < v_j < 1$ for each $j = 1, \ldots, n$. 
The geometric random variables are defined as
$P(g_{jk} = u) = (1-v_j v_{n-k+1}) (v_j v_{n-k+1})^{u}$ for all $u \geq 0$.

The initial state is $G_1^{pl} = g_{n1}$. 
The connection to point-to-line last passage percolation is that the largest particle at time $n$ has the representation 
\begin{equation*}
G_n^{pl}(n) = \max_{\pi \in \Pi_{n}} \sum_{(i, j) \in \pi} g_{ij}
\end{equation*}
where $\Pi_{n}^{\text{flat}}$ is the set of directed up-right paths nearest neighbour paths from $(1, 1)$ to the line 
$\{(i, j):i+j = n+1\}$. 
Moreover, $\mathbf{G}^{pl}(n)$ is the vector on the right hand side of Theorem \ref{main_theorem}.
The advantage of this interpretation is that the transition probabilities of $(\mathbf{G}^{pl}(k))_{1 \leq k \leq n}$ have a determinantal
form and this can be used to find the probability mass function of $\mathbf{G}^{pl}(n)$ as a determinant.

In the context of point-to-point last passage percolation the transition kernel of a Markov chain analogous to the above
is given in Theorem 1 of \cite{dieker2008}. This can be used to describe the update rule of
$\mathbf{G}^{pl}(n-1)$ to $\mathbf{G}^{pl}(n)$ from time $n-1$ to time $n$ by viewing $\mathbf{G}^{pl}(n-1)$ as being extended to an $n$-dimensional 
vector with zero as the leftmost position. 
We first define: for functions $f : \mathbb{Z} \rightarrow \mathbb{R}$ with $f(u) = 0$ for all $u < 0$,
\begin{equation*}
D^{(v)}f(u) = f(u) - vf(u-1), \qquad I^{(v)}f(u) = \sum_{j=0}^u v^{u-j} f(j) \text{ for } u \geq 0
\end{equation*}
and $I^{(v)} f(u) = 0$ for $u < 0$.
Suppose $p^{-1} = (1/p_1, \ldots, 1/p_n)$ and 
for a function $g : \mathbb{Z} \rightarrow \mathbb{R}$ with $g(u) = 0$ for $u < 0$ define
\begin{equation}
\label{function_integral_derivatives}
g_{p^{-1}}^{(ij)}(u) = \begin{cases} 
D^{(1/p_{i+1} \ldots 1/p_j)} g(u) & \text{ for } j > i \\
I^{(1/p_{j+1} \ldots 1/p_i)} g(u) & \text{ for } j < i \\
g(u) & \text{ for } j = i. 
\end{cases}
\end{equation}

\begin{lemma}[Dieker, Warren \cite{dieker2008}]
\label{lpp_transition_density}
As above, suppose the geometric random variables $g_{1, n-j+1}$ used in the update rule \eqref{update_rule_lpp} from $\mathbf{G}^{\text{pl}}(n-1)$ 
to $\mathbf{G}^{\text{pl}}(n-1)$ have parameters $1-v_1 v_j$ for $j = 1, \ldots, n$.
Then
\begin{multline*}
P(\mathbf{G}^{pl}(n) = (y_1, \ldots, y_n) \vert \mathbf{G}^{pl}(n-1) = (x_2, \ldots, x_n)) \\
= \prod_{k=1}^n (1- v_1 v_k) (v_1 v_k)^{y_k - x_k} \text{det}(w_{1, (1/(v_1 v_k))_{k=1}^n}^{(ij)}(y_j - x_i + j - i))_{i, j = 1}^n
\end{multline*}
where $x_1 := 0$ and $w_1(u) = 1_{\{u \geq 0\}}$.
\end{lemma}
The proof uses the RSK correspondence; a more direct proof is given in the case with all parameters equal in \cite{johansson2010} and with the geometric replaced by exponential data in \cite{FW}.

We will iteratively apply these one-step updates and use the following lemma to find the probability 
mass function for $\mathbf{G}^{pl}(n)$ as a single determinant.
\begin{lemma}
\label{ibp_lpp}
Suppose that $p = (p_1, \ldots, p_n)$ and $p_i > 0$ for $i = 1, \ldots, n$.
Let $(f_i)_{i = 1}^n$ be a collection of functions from $\mathbb{Z}_{\geq 0} \rightarrow \mathbb{R}$ 
and $g : \mathbb{Z} \rightarrow \mathbb{R}$ with $g(u) = 0$ for all $u < 0$. Then
\begin{multline*}
\sum_{x \in W^n_{\geq 0}} \text{det}(D^{(1/p_2 \ldots 1/p_j)} f_i(x_j + j - 1))_{i, j = 1}^n 
\text{det}( g_{1, p^{-1}}^{(ij)}(y_j - x_i + j-i))_{i, j = 1}^n \\
= \text{det}\left( D^{(1/p_2 \ldots 1/p_j)} \sum_{u \geq 0} 
f_i(u) g_1(y_j + j - 1 - u)\right)_{i, j = 1}^n. 
\end{multline*}
\end{lemma} 

\begin{proof}
We have $g(u) = 0$ for all $u < 0$ which means that $(J^{(p)}g)(z - \cdot)(u) = (I^{(1/p)} g)(z - u)$. 
We apply Lemma \ref{ibp_lemma} part (ii) with the functions $g_j(\cdot) = D^{(1/p_2 \ldots 1/p_j)} g(y_j + j - 1 - \cdot)$
where $D^{\emptyset} = \text{Id}$. We note that as $g$ is zero in a neighbourhood of infinity the condition
on the growth of $f$ can be omitted.
\end{proof}

\subsection{Proof of Theorem \ref{main_theorem}}

\begin{proof}
We prove by induction on $n$ that the probability mass function for $\mathbf{G}^{\text{pl}}(n)$ is given by $\pi(x_1, \ldots, x_n)$ and note that the case $n=1$ holds. 
The proposed probability mass function for  $\mathbf{G}^{pl}(n-1)$
is \begin{equation*}
\pi(x_2, \ldots, x_n) = c_{n-1} \prod_{k=2}^n v_k^{x_k} \text{det}(D_{x_j}^{(1/v_2 \ldots 1/v_j)} \phi_i(x_j + j - 2))_{i, j = 2}^n  
\end{equation*}
where $\phi_i(u) = v_i^{-(u+1)} - v_i^{u+1}$ for each $i = 2, \ldots, n$
and $c_{n-1} = (\prod_{2 \leq i < j \leq n} (v_i - v_j))^{-1} \prod_{j = 2}^n v_j^{n-1}$. 

We define
\begin{equation}
\label{pi_hat_defn}
\hat{\pi}(x_1, \ldots, x_n) = c_{n-1} \prod_{k=2}^n v_k^{x_k} \text{det}(D_{x_j}^{(1/v_2 \ldots 1/v_j)} \hat{\phi}_i(x_j + j - 2))_{i, j = 1}^n  
\end{equation}
where $D^{\emptyset} = \text{Id}$,  $\hat{\phi}_1(u) = 1_{\{u = -1\}}$ and $\hat{\phi}_i = \phi_i$ for each $i = 2, \ldots, n$. 
We first show that for $(x_1, \ldots, x_n) \in W^n_{\geq 0}$,
\begin{equation}
\label{pi_hat_equality}
\hat{\pi}(x_1, \ldots, x_n) = 1_{\{x_1 = 0\}} \pi(x_2, \ldots, x_n).
\end{equation}

Consider a Laplace expansion of $\hat{\pi}$ where the summation is indexed by a permutation $\sigma$. 
If $\sigma(1) = 1$ then the top-left entry in the matrix defining $\hat{\pi}$ is given by $\hat{\phi}_1(x_1 - 1)$ which equals $1$ if $x_1=0$ and $0$ otherwise. Therefore the terms in the Laplace expansion with $\sigma(1) = 1$ will give the desired expression for $\hat{\pi}$ 
and we need to show the remaining terms in the Laplace expansion of $\hat{\pi}$ are zero. 

Let $\sigma(1) = j$ for some $2 \leq j \leq n$ and $\sigma(i) = 1$ for some $2 \leq i \leq n$.  
For any $2 \leq j \leq n$, the $(1, j)$ entry in the matrix in \eqref{pi_hat_defn}  is only non-zero if $x_j = 0$ by using the definition of $\hat{\phi}_1$. 
On the other hand, the $(i, 1)$ entry in the matrix in \eqref{pi_hat_defn} 
is given by $\hat{\phi}_i(x_1 - 1) = v_i^{x_1} - v_i^{-x_1} = 0$ if $x_1 = 0$. Therefore as $x_1 \leq x_j$ all terms
in the Laplace expansion of $\hat{\pi}$ with $\sigma(1) \neq 1$ are zero. This proves $\eqref{pi_hat_equality}$.

We use \eqref{pi_hat_defn}, \eqref{pi_hat_equality} and the update rule in Lemma \ref{lpp_transition_density} to 
find the probability mass function for $\mathbf{G}^{pl}(n)$ as
\begin{multline}
\label{G_density_formula1}
c_{n-1} \sum_{(x_2, \ldots, x_n) \in W^{n-1}_{\geq 0}} \prod_{k=2}^n v_k^{x_k} \text{det}(D_{x_j}^{(1/v_2 \ldots 1/v_j)} \hat{\phi}_i(x_j + j - 2))_{i, j = 1}^n\\
\cdot \prod_{k=1}^n (1- v_1 v_k) (v_1 v_k)^{y_k - x_k} \text{det}(w_{1, (1/(v_1 v_k))_{k=1}^n}^{(ij)}(y_j - x_i +j - i))_{i,j=1}^n
\end{multline}
where $x_1:=0$. 
We use the identities: 
\begin{equation*}
D^{(\alpha)}f(u) = \beta^u D^{(\alpha/\beta)}(f(u) \beta^{-u}), \qquad \quad I^{(\alpha)}f(u) = \beta^u I^{(\alpha/\beta)}(f(u) \beta^{-u})
\end{equation*}
to show that with $u = y_j - x_i + j - i$,
\begin{gather}
\label{change_parameters1}
v_1^{u}  I^{(1/(v_1 v_{j+1}), \ldots, 1/(v_1 v_i))}
(w_1(u))  =   I^{(1/v_{j+1}, \ldots, 1/v_i)}(
w_1(u) v_1^{u}) \text{ for } i > j \\
\label{change_parameters2}
v_1^{u} D^{(1/(v_1 v_{i+1}), \ldots, 1/(v_1 v_j))} (w_1(u)) = D^{(1/v_{i+1}, \ldots, 1/v_j)}(w_1(u) v_1^{u})
\text{ for } j > i.
\end{gather}
We use \eqref{change_parameters1} and \eqref{change_parameters2} in 
\eqref{G_density_formula1} to obtain the probability mass function for $\mathbf{G}^{pl}(n)$
as
\begin{multline*}
c_{n-1} \sum_{(x_2, \ldots, x_n) \in W^{n-1}_{\geq 0}} \prod_{k=1}^n v_k^{y_k} \text{det}\left(D_{x_j}^{(1/v_2 \ldots 1/v_j)} \hat{\phi}_i(x_j + j - 2)\right)_{i, j = 1}^n \\
\cdot \prod_{k=1}^n (1- v_1 v_k) \text{det}(\hat{w}_{1, v^{-1}}^{(ij)}(y_j - x_i +j - i))_{i,j=1}^n
\end{multline*}
where $\hat{w}_1(u) = w_1(u) v_1^u$ and $\hat{w}_{1, v^{-1}}^{(ij)}$ is defined by \eqref{function_integral_derivatives}.
We apply Lemma \ref{ibp_lpp} to show this equals
\begin{equation}
\label{G_formula_2}
c_{n-1} \prod_{k=1}^n (1- v_1 v_k) v_k^{y_k} \text{det}\left(D_{y_j}^{(1/v_2 \ldots 1/v_j)} 
\sum_{u = 0}^{y_j + j -1} \hat{\phi}_i(u-1) \hat{w}_1(y_j - u + j - 1) \right)_{i, j = 1}^n.
\end{equation}
In the case $i = 1$, recall $\phi_1(x) = v_1^{(x+1)} - v_1^{-(x+1)}$ and observe that 
\begin{equation}
\label{row_1_simplifications}
 \sum_{u = 0}^{y_j + j -1} \hat{\phi}_1(u-1) v_1^{y_j - u + j - 1} = v_1^{y_j + j - 1} = \frac{v_1}{1 - v_1^2} D^{(1/v_1)} \phi_1(y_j + j - 1).
 \end{equation}
 In the case $2 \leq i \leq n$ observe that 
 \begin{IEEEeqnarray}{rCl}
 \label{row_2_simplifications}
 \sum_{u = 0}^{y_j + j - 1} \phi_i(u - 1) v_1^{y_j - u + j - 1} & = &  \frac{v_i^{y_j + j}}{v_1 - v_i} -
  \frac{v_1^{y_j + j - 1}}{1-v_i/v_1}
 + \frac{v_1^{y_j + j -1}}{1 - 1/(v_1 v_i)} - \frac{v_i^{-y_j - j + 1}}{v_1 v_i - 1} \IEEEnonumber \\
 & = & \frac{v_i v_1}{(1-v_1 v_i)(v_1 - v_i)} D^{(1/v_1)} \phi_i(y_j + j - 1) + C v_1^{y_j}
 \end{IEEEeqnarray}
 where $C$ is independent of $y_j$.  
 Using \eqref{row_1_simplifications} and \eqref{row_2_simplifications} in \eqref{G_formula_2} and removing the terms $C v_1^{y_j}$ 
 using row operations we find 
 \begin{equation*}
 c_{n-1}\frac{v_1}{1 - v_1^2} \prod_{j = 2}^n  \frac{v_1 v_j}{(1 - v_1 v_j )(v_1 - v_j)}  \prod_{k = 1}^n (1- v_k v_1)  v_k^{y_k}  \text{det}(D_{y_j}^{(1/v_1 \ldots 1/v_j)} \phi_i(y_j + j - 1))_{i, j = 1}^n.  
 \end{equation*}
 The prefactor equals $c_n$ where $c_{n} = \prod_{1 \leq i < j \leq n}\frac{1}{(v_i - v_j)} \prod_{j = 1}^n v_j^n$ and so we establish inductively that the 
 probability mass function of $\mathbf{G}^{pl}(n)$ 
 is given by 
\begin{equation}
\label{lpp_prob_mass}
\pi(y_1, \ldots, y_n) = c_n \prod_{k=1}^n v_k^{y_k} \text{det}( D_{y_j}^{(1/v_1 \ldots 1/v_j)} \phi_i(y_j + j - 1))_{i, j = 1}^n.
\end{equation}
We recall that in Proposition \ref{Invariant_measure}, we deferred the proof of positivity of $\pi$ and the 
normalisation constant. This is now proven as 
we have identified $\pi$ as the probability mass function of $\mathbf{G}^{pl}(n)$. 
Moreover, Equation \eqref{lpp_prob_mass} and Proposition \ref{Invariant_measure} proves Theorem \ref{main_theorem} when $v_1, \ldots, v_n$ are distinct. The distribution of 
$(Y_1^*, \ldots, Y_n^*)$ is continuous in $(v_1, \ldots, v_n)$ on the set $(0, 1)^n$ from Lemma \ref{continuity_invariant_measure} and the distribution of $(G(1, n), \ldots, G(1, 1))$ is continuous in $(v_1, \ldots, v_n)$ on the same set as a finite number of operations of summation and maxima applied to geometric random 
variables. This
completes the proof of Theorem \ref{main_theorem}.
\end{proof}

\subsection{The largest particle}

\begin{proposition}
Let $F(\eta) = P(Y_n^* \leq \eta) = P(G(1, 1) \leq \eta).$
For distinct $0 < v_1, \ldots, v_n < 1$, 
\begin{equation*}
F(\eta) =  \prod_{1 \leq i \leq j \leq n} (1 - v_i v_j) \left(\prod_{j = 1}^n v_j\right)^{\eta} \text{Sp}_{\eta^{(n)}}(v_1, \ldots, v_n)
\end{equation*} 
where $\text{Sp}$ denotes the symplectic Schur function from \eqref{symplectic_schur_defn}, and $\eta^{(n)}$ denotes an $n$-dimensional vector $(\eta, \ldots, \eta)$. 
\end{proposition}
For point-to-line last passage percolation this was proven in \cite{Bisi_Zygouras2} and related to earlier formulas 
for point-to-line last passage percolation in \cite{baik_rains} and \cite{bisi_zygouras}. We could appeal to this and Theorem \ref{top_particle_theorem}
to prove the same expression for the distribution function of $Y_n^*$. We now show that it follows
quickly from Proposition \ref{Invariant_measure}. 
\begin{proof}
From Proposition \ref{Invariant_measure} we have
\begin{equation*}
F(\eta) = c_n \sum_{y_1, \ldots, y_n \leq \eta, y \in W^n_{\geq 0}} \prod_{k=1}^n v_k^{y_k}
\text{det}(D^{(1/v_1 \ldots 1/v_j)} \phi_i(y_j + j - 1))_{i, j = 1}^n.  
\end{equation*}
We perform the summation in $y_n$ from $y_{n-1}$ to $\infty$ which replaces the last column by 
$v_n^{\eta} D^{(1/v_1 \ldots 1/v_{n-1})} \phi_i(\eta +n - 1) - v_n^{y_{n-1} -1} 
D^{(1/v_1 \ldots 1/v_{n-1})} \phi_i(y_{n-1} +n - 2)$. The second term differs from the penultimate column 
by a factor of $v_{n}^{y_{n-1} -1} v_{n-1}^{-y_{n-1}}$ which is non-zero and independent of $i$. Therefore the second term 
can be removed from the last column by column operations. We now apply this procedure inductively in order $x_{n-1}, \ldots, x_1$ to obtain 
\begin{equation*}
F(\eta) = c_n  \left(\prod_{k=1}^n v_k \right)^{\eta} \text{det}(D^{(1/v_2 \ldots 1/v_{n-1})} \phi_i(\eta + j - 1))_{i, j = 1}^n
\end{equation*}
where $D^{\emptyset} = \text{Id}$. 
We relate this to a symplectic Schur function by using column operations. 
The entries in the first column are $\phi_i(\eta) = v_i^{-(\eta + 1)} - v_i^{\eta + 1}$. In the second column, 
the entries are $D^{(1/v_2)} \phi_i(\eta + 1) = (v_{i}^{-(\eta + 2)} - v_i^{\eta + 2}) - v_2^{-1} (v_i^{-(\eta+1)}
 - v_i^{\eta + 1})$ and the second bracketed term can be removed by column operations. This can be continued inductively and leads to 
 \begin{IEEEeqnarray*}{rCl}
 F(\eta) & = & c_n (-1)^n \left(\prod_{k=1}^n v_k\right)^{\eta} \text{det}(v_i^{\eta + j} - v_i^{-(\eta + j)})_{i, j = 1}^n \\
 & = &  c_n \left(\prod_{k=1}^n v_k\right)^{\eta} Sp_{\eta^{(n)}}(v_1, \ldots, v_n) \text{det}(v_i^j - v_i^{-j})_{i, j = 1}^n.
 \end{IEEEeqnarray*}
The proof is now completed by using \eqref{fulton_harris_eqn} to equate the normalisation constants.
\end{proof}

\section{A dynamically reversible process}
 \label{dynamic_reversibility}

We will suppose throughout that $0 < v_1, \ldots, v_n < 1$. 
Let $e_{ij}$ denote the vector taking value $1$ in position $(i, j)$ and $0$ otherwise. Let $S = \{(i, j):  i, j \in \mathbb{Z}_{\geq 1}, i + j \leq n+1\}$ and 
\begin{equation*}
\mathcal{X} = \{ (x_{ij})_{(i, j ) \in S} : x_{ij} \in \mathbb{Z}_{\geq 0}, x_{i+1, j} \leq x_{ij} \text{ and } x_{i, j+1} \leq x_{ij} \}.
\end{equation*} 
We define a continuous-time Markov process $(X_{ij}(t) : i+j \leq n+1, t \geq 0)$ taking values in $\mathcal{X}$ by specifying its transition rates.
For $x, x + e_{ij} +  e_{i j-1} + \ldots + e_{ik} \in \mathcal{X}$, $(i, j), (i, k) \in S$ and $k \leq j$ define
\begin{equation}
\label{X_defn1}
q(x, x + e_{ij} + e_{i j-1} + \ldots + e_{ik}) = v_{n-j + 1} (v_{n-j + 1} v_{i-1})^{-1_{\{x_{ij} \geq x_{i-1, j+1}\}}} 1_{\{x_{ij} = x_{i j-1} = \ldots = x_{ik}\}}
\end{equation}
with the notation that $x_{0, j} = \infty$ for $j = 2, \ldots, n+1$.
For $x, x - e_{ij} -  e_{i+1 j} - \ldots - e_{lj} \in \mathcal{X}, (i, j), (l, j) \in S$ and $l \geq i$ define
\begin{equation}
\label{X_defn2}
q(x, x - e_{ij} -  e_{i+1 j} - \ldots - e_{lj}) = v_{n - j +1}^{-1} (v_{i-1} v_{n - j + 1})^{1_{\{x_{ij} > x_{i - 1, j+1}\}}} 1_{\{x_{ij} = x_{i+1 j} = \ldots = x_{lj}\}}.
\end{equation}
All other transition rates are zero. The fact that $q(x, x') \neq 0$ only 
if $x, x' \in \mathcal{X}$ corresponds to blocking interactions.
This defines a multi-dimensional Markov chain with interactions shown in Figure \ref{fig_Xarray}, where the arrows in Figure \ref{fig_Xarray} correspond to the following interactions. 
\begin{enumerate}[(i)]
\item A push-block interaction denoted $A \rightarrow B$. If $A = B$ and $A$ jumps to the right by one then $B$ also jumps to the right by one (this may then causes further right jumps if $B \rightarrow C$). If $A = B$ and $B$ jumps left then this jump is suppressed. 
\item A push-block interaction denoted $A \downarrow B$ where $A$ is at the base of the arrow and $B$ at the head of the arrow in Figure \ref{fig_Xarray}. If $A = B$ and $A$ jumps to the left by one then $B$ also jumps to the left by one (this may then causes further left jumps if $B \downarrow C$). If $A = B$ and $B$ jumps right then this jump is suppressed. 
 \item An interaction $A \leadsto B$ in which the rates of right and left jumps experienced by $B$ depend on its location relative to $A$. The particular form is given in \eqref{X_defn1} and \eqref{X_defn2} and is chosen such that
 $(X_{ij}(t):i+j \leq n+1, t \geq 0)$ is dynamically reversible, see part (ii) of Theorem \ref{dynamical_reversibility_theorem}. 
 \item An interaction with a wall in which all left jumps below zero are suppressed. This is depicted by the diagonal 
 line on the left side of Figure \ref{fig_Xarray}.
 \end{enumerate}

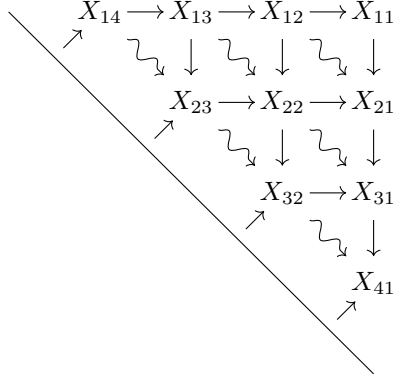
\begin{figure}
\centering
 \begin{tikzpicture}[scale = 1.2]
 \node at (4, 4) {$X_{11}$};
 \draw[->] (3.3, 4) -- (3.7, 4);
  \draw[->] (4, 3.7) -- (4, 3.3);
 \node at (3, 4) {$X_{12}$};
 \draw[->] (2.3, 4) -- (2.7, 4);
   \draw[->] (4, 2.7) -- (4, 2.3);
 \node at (2, 4) {$X_{13}$};
 \draw[->] (1.3, 4) -- (1.7, 4);
   \draw[->] (4, 1.7) -- (4, 1.3);
 \node at (1, 4) {$X_{14}$};
 \draw[->] (3.3, 3) -- (3.7, 3);
   \draw[->] (3, 3.7) -- (3, 3.3);
 \node at (4, 3) {$X_{21}$};
 \draw[->] (2.3, 3) -- (2.7, 3);
   \draw[->] (3, 2.7) -- (3, 2.3);
 \node at (3, 3) {$X_{22}$};
 \draw[->] (3.3, 2) -- (3.7, 2);
 \node at (2, 3) {$X_{23}$};
   \draw[->] (2, 3.7) -- (2, 3.3);
 \node at (4, 2) {$X_{31}$};
 \node at (3, 2) {$X_{32}$};
 \node at (4, 1) {$X_{41}$};
 \draw (0, 4) -- (4, 0);
 \path[draw = black, ->, snake it]    (3.3, 1.7) -- (3.7,1.3);
 \path[draw = black, ->, snake it]    (2.3, 2.7) --  (2.7,2.3);
 \path[draw = black, ->, snake it]    (1.3,3.7) -- (1.7, 3.3);
 \path[draw = black, ->, snake it]    (3.3,2.7) -- (3.7, 2.3);
 \path[draw = black, ->, snake it]    (3.3,3.7) -- (3.7, 3.3);
 \path[draw = black, ->, snake it]    (2.3,3.7) -- (2.7, 3.3);
  \draw[->] (0.6, 3.6) -- (0.8, 3.8);
  \draw[->] (1.6, 2.6) -- (1.8, 2.8);
   \draw[->] (2.6, 1.6) -- (2.8, 1.8);
    \draw[->] (3.6, 0.6) -- (3.8, 0.8);
 \end{tikzpicture}
\caption{The interactions in the system $\{X_{ij} : i + j \leq n+1\}$.}
\label{fig_Xarray}
\end{figure}

To find the invariant measure of $X$ we use a result which has found applications in the queueing theory literature.
\begin{lemma}[Theorem 1.13, Kelly \cite{kelly}]
\label{Kelly}
Let $X$ be a stationary Markov process with state space $E$ and transition rates $(q(j, k))_{i, j \in E}$. Suppose we can find positive sequences
$(\hat{q}(j, k))_{i, j \in E}$ and $(\pi(j))_{j \in E}$ with $\sum_{j \in E} \pi(j) = 1$ such that: 
\begin{enumerate}[(i)]
\item $q(j) = \hat{q}(j)$ where $q(j) = \sum_{k \neq j} q(j, k)$ and $\hat{q}(j) = \sum_{k \neq j} \hat{q}(j, k)$.
\item $\pi(j) q(j, k) = \pi(k) \hat{q}(k, j)$. 
\end{enumerate}
Then $\pi$ is the invariant measure for $X$ and $\hat{q}$ are the transition rates of the time reversal of $X$ in stationarity. 
\end{lemma}

The proof is straightforward:  using (ii) then (i)
\begin{equation*}
\sum_{j \in E} \pi(j) q(j, k) = \sum_{j \in E} \pi(k) \hat{q}(k, j) = \pi(k) \hat{q}(k) = \pi(k) q(k). 
\end{equation*}
Nonetheless this gives a convenient way of verifying an invariant measure \emph{if} we can guess 
the transition rates of the time reversed process. In general, this is an intractable problem. However, in this case we can 
make the choice that the invariant measure is a field of point-to-line last passage percolation times and the reversed transition
probabilities are given by \emph{reversing the direction of all interactions} between particles in Figure \ref{fig_Xarray}
and changing the order of the parameters $(v_1, \ldots, v_n) \rightarrow (v_n, \ldots, v_1)$ (the interactions 
with the wall remain unchanged). This is motivated by the construction of \cite{FW}.

More precisely, we define the reversed transition rates as follows.
For $x, x + e_{ij} + e_{i j -1} + \ldots + e_{ik} \in \mathcal{X}$ for $(i, j), (i, k) \in S$ and $k \leq i$ define
\begin{equation*}
\hat{q}( x + e_{ij} + e_{i j -1} + \ldots + e_{ik}, x) = v_i^{-1} (v_{n-k+2} v_i)^{1_{\{ x_{ik} \geq x_{i+1, k-1}\}}} 1_{\{x_{ij} = x_{i j - 1} = \ldots = x_{ik}\}}
\end{equation*} 
with the notation that $x_{j0} = \infty$ for $j = 2, \ldots, n+1$.
For $x, x -e_{ij} - e_{i+1, j} - \ldots - e_{lj} \in \mathcal{X}$, $(i, j), (l, j) \in S$ and $l \geq i$, define 
\begin{equation*}
\hat{q}(x -e_{ij} - e_{i+1, j} - \ldots - e_{lj}, x) = v_l (v_{n-j+2} v_l)^{-1_{\{x_{lj} > x_{l+1, j-1} \}}} 1_{\{x_{ij} = x_{i+1 j} = \ldots = x_{lj}\}}.
 \end{equation*}

Our proposed invariant measure is the probability mass function of $(G(i, j): i+j \leq n+1)$. This has an explicit form: 
\begin{equation*}
\pi(x) = \prod_{\{i + j < n+1\}} (1- v_i v_{n-j+1}) (v_i v_{n-j+1})^{x_{ij} - \max(x_{i+1 j}, x_{i j+1})} 
\prod_{i=1}^n (1 - v_i^2) v_i^{2 x_{i n-i+1}} 
\end{equation*}

\begin{theorem}
\label{dynamical_reversibility_theorem}
Suppose that $0 < v_1, \ldots, v_n <1$.
Let $(X_{ij}^{(v_1 \ldots v_n)}(t) : i + j \leq n+1, t \geq 0)$ be the continuous-time Markov process with transition rates given 
by \eqref{X_defn1} and \eqref{X_defn2}. This process has a unique invariant measure $(X_{ij}^*: i + j \leq n+1)$
which satisfies
\begin{equation*}
(X_{ij}^*: i + j \leq n+1) \stackrel{d}{=} (G(i, j) : i+j \leq n+1).
\end{equation*}
When run in stationarity, 
\begin{equation*}
(X_{ij}^{(v_1, \ldots, v_n)}(t))_{t \in \mathbb{R}, i+j \leq n+1} \stackrel{d}{=} (X_{ji}^{(v_n, \ldots, v_1)}(-t))_{t \in \mathbb{R}, i+j \leq n+1}. 
\end{equation*}
\end{theorem}

The process $(X_{ij}(t): i + j \leq n+1, t \geq 0)$ is irreducible, does not explode and 
has a unique invariant measure.
The second statement is the statement that when run in stationarity $(X_{ij}(t): i + j \leq n+1, t \geq 0)$ is dynamically reversible. 

\begin{proof}
We will use Lemma \ref{Kelly}.
We first prove that for all $x, x' \in \mathcal{X}$, 
\begin{equation}
\label{dynamical_rev_eq1}
\pi(x) q(x, x') = \pi(x') \hat{q}(x', x). 
\end{equation}

First consider the case when $x' = x + e_{ij} + e_{i j -1} + \ldots + e_{ik}$. Both sides are zero unless 
$x_{ij} = x_{i j-1} = \ldots = x_{ik}$. 
When these equalities hold, $\max(x_{i p}, x_{i-1, p+1}) = x_{i-1, p+1}$ for each $p = j-1, \ldots, k$ and $\max(x_{i+1, p-1}, x_{i p}) = x_{i p}$ for each $p = j, \ldots, k+1$. 
Therefore when $x_{ij} = x_{i j-1} = \ldots = x_{ik}$, 
\begin{equation*}
\pi(x) = \bar{\pi}_1 (v_i v_{n-j+1})^{x_{ij}} (v_{i-1} v_{n-j+1})^{-x_{ij} 1_{\{x_{ij} \geq x_{i-1, j+1}\}}} (v_i v_{n-k+2})^{-x_{ik} 1_{\{ x_{ik} \geq x_{i+1, k-1}\}}}
\end{equation*}
where $\bar{\pi}_1$ does not depend on $x_{ij}, x_{i j-1}, \ldots, x_{ik}$. 
In particular, with $x' = x + e_{ij} + e_{i j -1} + \ldots + e_{ik}$, 
\begin{equation*}
\frac{\pi(x')}{\pi(x)} = v_i v_{n-j+1} (v_{i-1} v_{n-j+1})^{-1_{\{x_{ij} \geq x_{i-1, j+1} \}}} (v_i v_{n-k+2})^{-1_{\{ x_{ik} \geq x_{i+1, k-1}\}}}.
\end{equation*}
We compare to the ratio 
\begin{equation*}
\frac{q(x, x')}{\hat{q}(x', x)} = 
v_i v_{n-j+1} (v_{i-1} v_{n-j+1})^{-1_{\{x_{ij} \geq x_{i-1, j+1} \}}} (v_i v_{n-k+2})^{-1_{\{ x_{ik} \geq x_{i+1, k-1}\}}}.
\end{equation*}
Combining the above two equations proves \eqref{dynamical_rev_eq1} in the case  $x' = x + e_{ij} + e_{i j -1} + \ldots + e_{ik}$.
The second case is when $x' = x -e_{ij} - e_{i+1, j} - \ldots - e_{lj}$ and proceeds in a similar manner. 
Both sides are zero unless $x_{ij} = x_{i+1 j} = \ldots = x_{lj}$. When these equalities hold, then $\max(x_{p-1 j}, x_{p, j-1}) = x_{p, j-1}$ for $p = l, \ldots, i+1$ and $\max(x_{p-1, j+1}, x_{p, j}) = x_{p j}$ for $p = l, \ldots, i+1$. 
Therefore 
\begin{equation*}
\pi(x) = \bar{\pi}_2 (v_l v_{n-j+1})^{x_{lj}} (v_l v_{n-j+2})^{-x_{lj} 1_{\{x_{lj} > x_{l+1, j-1}\}}} 
(v_{i-1} v_{n-j+1})^{-x_{ij} 1_{\{x_{ij} > x_{i-1, j+1}\}}}
\end{equation*}
where $\bar{\pi}_2$ does not depend on $x_{ij}, x_{i+1, j}, \ldots, x_{lj}$. 
Therefore letting
$x' = x-e_{ij} - e_{i+1, j} - \ldots - e_{lj}$ we have
\begin{equation*}
\frac{\pi(x)}{\pi(x')} = (v_l v_{n-j +1})(v_l v_{n-j+2})^{-1_{\{x_{lj} > x_{l+1, j-1} \}}} (v_{i-1} v_{n- j +1})^{-1_{ \{ x_{ij} > x_{i-1, j+1} \}}}
\end{equation*}
and 
\begin{equation*}
\frac{\hat{q}(x', x)}{q(x, x')}
= (v_l v_{n-j +1})(v_l v_{n-j+2})^{-1_{\{x_{lj} > x_{l+1, j-1} \}}} (v_{i-1} v_{n-j +1})^{-1_{ \{ x_{ij} > x_{i-1, j+1} \}}}.
\end{equation*}
 The two above equations prove \eqref{dynamical_rev_eq1} in the case when $x' = x -e_{ij} - e_{i+1, j} - \ldots - e_{lj}$. 
 Both sides of \eqref{dynamical_rev_eq1} are zero in all other cases and so we have proven \eqref{dynamical_rev_eq1}.

We now show that for all $x \in \mathcal{X}$ we have $q(x) = \hat{q}(x)$. 
This follows from comparing,
\begin{multline}
\label{q_rate}
q(x) = v_1 + v_1^{-1} 1_{\{x_{1n} > 0\}} + \sum_{k=1}^{n-1} v_{n-k+1} + v_{n-k+1}^{-1} 1_{\{x_{1 k} > x_{1 k+1}\}} \\
+ \sum_{\{i \neq 1, i+j \leq n+1\}} v_{n-j+1} (v_{n-j+1} v_{i-1})^{-1_{\{x_{ij} \geq x_{i-1, j+1}\} }} 1_{\{x_{ij} < x_{i-1, j}\}} \\ +
 \sum_{\{i \neq 1, i+j < n+1\}} v_{n-j+1}^{-1} (v_{n-j+1} v_{i-1})^{1_{\{x_{ij} > x_{i-1, j+1} \}}} 1_{\{x_{ij} > x_{i j+1}\}} \\
 +  \sum_{\{i \neq 1, i+j = n+1\}} v_{n-j+1}^{-1} (v_{n-j+1} v_{i-1})^{1_{\{x_{ij} > x_{i-1, j+1} \}}} 1_{\{x_{ij} > 0\}}
\end{multline}
and
\begin{multline}
\label{q_hat_rate}
\hat{q}(x) = \sum_{k=1}^{n-1} v_k + v_k^{-1} 1_{\{x_{k1} > x_{k+1, 1}\}} + v_n + v_n^{-1} 1_{\{x_{n1} > 0 \}} \\
+ \sum_{\{i \neq 1, i+j \leq n+1\}} 
v_{i-1}(v_{i-1} v_{n-j+1})^{-1_{\{x_{i-1, j+1} \geq x_{ij}\}}} 1_{\{x_{i-1, j+1} < x_{i-1, j}\}} \\ +  \sum_{\{i \neq 1, i+j < n+1\}} 
v_{i-1}^{-1}(v_{i-1} v_{n-j+1})^{1_{\{x_{i-1, j+1} > x_{ij}\}}} 1_{\{x_{i-1, j+1} > x_{i, j+1}\}}
\\ +\sum_{\{i \neq 1, i+j = n+1\}} 
v_{i-1}^{-1}(v_{i-1} v_{n-j+1})^{1_{\{x_{i-1, j+1} > x_{ij}\}}} 1_{\{x_{i-1, j+1} > 0\}}.
\end{multline}
One way to check that $q(x) = \hat{q}(x)$ is to check the equality first in the case when the inequalities 
$x_{ij} < x_{i-1, j}$ for each $i \neq 1$, $x_{ij} < x_{i j-1}$  
for each $j \neq 1$ and $x_{i n-i+1} > 0$ hold for each $i = 1, \ldots, n$. This case can be seen directly from \eqref{q_rate} and \eqref{q_hat_rate}. We now consider the rates of jumps which are suppressed in each case when these inequalities no longer hold:
\begin{enumerate}[(i)]
\item If $x_{i, n-i+1} = 0$ then both forwards and backwards in time a jump of rate $v_{i}^{-1}$ is suppressed by the wall.
\item If $x_{ij} = x_{i, j-1}$ then forwards in time the left jump of the $(i, j-1)$ particle is suppressed and the suppressed jump has rate 
$v_{n-j+2}^{-1}$ because $x_{i, j - 1} = x_{ij} \leq x_{i-1, j}$. 
Backwards in time, the right jump of the $(i, j)$ particle is suppressed and the suppressed jump has rate $v_{n-j+2}^{-1}$ 
because $x_{i, j} = x_{i, j-1} \geq x_{i+1, j-1}$.
\item If $x_{ij} = x_{i-1 j}$ then forwards in time the right jump of the $(i, j)$ particle is suppressed and the suppressed jump has rate 
$v_{i-1}^{-1}$ because $x_{ij} = x_{i-1, j} \geq x_{i-1, j+1}$. 
Backwards in time, the left jump of the $(i-1, j)$ particle is suppressed and the suppressed jump has rate $v_{i-1}^{-1}$ 
because $x_{i-1 j} = x_{ij} \leq x_{i j-1}$. 
\end{enumerate}
Using Lemma \ref{Kelly}, we have now established that $\pi$ is the invariant measure and $\hat{q}$ are the reversed 
transition rates in stationarity of $(X_{ij}(t):i+j \leq n+1, t \geq 0)$.  The second statement in the Theorem follows from comparing $q$ and $\hat{q}$ and observing that they are identical after 
the swap $x_{ij} \rightarrow x_{ji}$ and $(v_1, \ldots, v_n) \rightarrow (v_n, \ldots, v_1)$.  
\end{proof}

We end by discussing two further properties of the the process $X$. These properties can both be proved by running the process $(X_{ij}(t) : i + j \leq n+1, t \geq 0)$ in stationarity, forwards and backwards in time, and follow in exactly the same way as Section 5 of \cite{FW} as they depend on the structural properties of the $X$ array rather than the exact dynamics. 
\begin{enumerate}[(i)]
\item The marginal distribution of any row $(X_{i, n-i+1}, \ldots, X_{i, 1})$ run forwards in time is PushASEP with a wall
with rate vector $(v_i, \ldots, v_n)$. The marginal distribution of any column $(X_{n-j+1, j}, \ldots, X_{1, j})$ run backwards in time is PushASEP with 
a wall with rate vector $(v_{n-j+1}, \ldots, v_1)$ 
\item Let $Q^n_t$ denote the transition semigroup for PushASEP with a wall with $n$ particles.
Let $P_{n-1 \rightarrow n}$ denote the transition kernel for the update of the Markov chain $\mathbf{G}^{\text{pl}}$ defined in 
Section \ref{point_to_line} from time $n-1$ to $n$. There is an intertwining
between $Q^{n-1}_t$ and $Q^n_t$ with intertwining kernel given by $P_{n-1 \rightarrow n}$. 
In operator notation, 
\begin{equation*}
Q_t^{n-1} P_{n-1 \rightarrow n} = P_{n-1 \rightarrow n} Q_t^n.
\end{equation*}
\end{enumerate}

\section{Push-block dynamics and Proposition \ref{two_sided_intertwining}}
\label{Intertwining}

The aim of this Section is to describe how Proposition \ref{two_sided_intertwining}
can be obtained by a construction of an interacting particle system with pushing and blocking interactions.  This section is adapting the proof of Theorem 2.1 in \cite{warren_windridge} with a
different intertwining \eqref{intertwining} replacing Equation 3.3 from \cite{warren_windridge}.

We follow the set-up and notation of \cite{warren_windridge}.
For each $n \geq 1$, let $(\mathfrak{X}(t) : t \geq 0)$ be a continuous-time Markov process 
$\mathfrak{X}(t) = (\mathfrak{X}^j_i(t))_{1\leq i \leq j \leq n}$ taking values in
$\mathbb{K}_n = \{(x^j_i)_{1 \leq i \leq j \leq n} \text{ with } x^{j+1}_i \leq x^j_i \leq x^{j+1}_{i+1}\}.$
We use $x^j$ to denote the vector $x^j = (x^j_1, \ldots, x^j_j)$ and describe $\mathfrak{X}^j$ 
as the positions of the particles in the $j$-th level of $\mathfrak{X}$. Let $v_i > 0$ for each $i \geq 1$. 

The dynamics of $\mathfrak{X}$ is governed by $n(n+1)$ independent exponential clocks, where each particle
in the $j$-th level has two independent exponential clocks with rates $v_j$ and $v_{j}^{-1}$ corresponding to its right and left jumps  
respectively. When the clock of a particle in the $j$-th level rings, that particle attempts to jump to the right or left but 
will experience a pushing and a blocking interaction which ensures that $\mathfrak{X}$ remains within $\mathbb{K}_n$. In summary, a particle at the $j$-th level \emph{pushes} particles at levels $k > j$ and is \emph{blocked}
by particles at levels $k < j$.
More precisely, suppose the right clock of $\mathfrak{X}^j_i$ rings. 
\begin{enumerate}[(i)]
\item If $\mathfrak{X}^j_i = \mathfrak{X}^{j-1}_i$ then the right jump is suppressed. 
\item If $\mathfrak{X}^j_i < \mathfrak{X}^{j-1}_i$ 
and $\mathfrak{X}^j_i = \mathfrak{X}^{j+1}_{i+1}$ then $\mathfrak{X}^j_i$ jumps right by one and \emph{pushes} 
$\mathfrak{X}^{j+1}_{i+1}$ to the right by one. The right jump of $\mathfrak{X}^{j+1}_{i+1}$ may then cause further right jumps in the same way. 
\item In all other cases $\mathfrak{X}^j_i$ jumps to the right by one and all other 
particles are unchanged. 
\end{enumerate}
If the left clock of $\mathfrak{X}^j_i$ rings then we have the same trichotomy of cases: (i) if $\mathfrak{X}^j_i = \mathfrak{X}^{j-1}_{i-1}$ then the left
jump of $\mathfrak{X}^j_i$ is suppressed ; (ii) if $\mathfrak{X}^j_i > \mathfrak{X}^{j-1}_{i-1}$ and  $\mathfrak{X}^j_i = \mathfrak{X}^{j+1}_i$ then $\mathfrak{X}^j_i$ jumps to the left and \emph{pushes} $\mathfrak{X}^{j+1}_i$ to the left by one which may then push further particles to the left; and (iii) in all other cases $\mathfrak{X}^j_i$ jumps to the left 
by one and all other 
particles are unchanged.

Let $n \geq 1$ and for $x \in \mathbb{K}_n$ let $w_v(x) = \prod_{i=1}^n v_i^{\lvert x^i \rvert - \lvert x^{i-1} \rvert}$ where $\lvert x \rvert = \sum_{j=1}^d x_j$
for $x \in \mathbb{R}^d$ and $\lvert x^0 \rvert = 0$.
For any $z  \in W^n$ we define $\mathbb{K}_n(z) = \{(x^j_i)_{1 \leq i \leq j \leq n}  \in K_n : x^n = z\}$ and a probability measure on $\mathbb{K}_n(z)$
by $M_z(x) = w_v(x)/S_z(v)$ for all $x \in \mathbb{K}_n(z)$.

\begin{proposition}
\label{GT_process}
Suppose $z \in W^n$ and that $(\mathfrak{X}(t): t \geq 0)$ has initial distribution $M_z(\cdot)$.  
Then $(\mathfrak{X}^n(t) : t \geq 0)$ is a Markov process with conservative $Q$-matrix, given for $x \in W^n$ 
by
\begin{equation*}
Q(x, x \pm e_i) = \frac{S_{x \pm e_i}(v)}{S_x(v)} 1_{\{x \pm e_i \in W^n\}}, \text{ for } i = 1, \ldots, n
\end{equation*}
with $Q(x, x) = -\sum_{i=1}^n (v_i^{-1} + v_i)$.
\end{proposition}

We prove this Proposition inductively in $n$ by analysing the two consecutive bottom layers of $\mathfrak{X}$
and include the statement that $Q$ is conservative as part of the induction argument.
For $x \in W^n$ and $y \in W^{n+1}$ we will write $x \preceq y$ to mean that 
$y_1 \leq x_1 \leq y_2 \leq \ldots x_n \leq y_{n+1}$ and define
$W^{n, n+1} = \{x \in W^n, y \in W^{n+1}: x \preceq y\}$. 
By the inductive hypothesis, the marginal distribution of the two consecutive bottom layers of $\mathfrak{X}$ is a continuous-time 
Markov process $(X(t), Y(t))_{t \geq 0}= (X_1(t), \ldots, X_n(t), Y_1(t), \ldots, Y_{n+1}(t))_{t \geq 0}$ taking values in $W^{n, n+1}$
and with $Q$-matrix given by the off-diagonal entries: for $(x, y), (x', y') \in W^{n, n+1}$, 
\begin{equation*}
\mathcal{A}((x, y), (x', y')) = \begin{cases}
Q_X(x, x + e_i) & \text{ if } (x', y') = (x + e_i, y) \text{ and } x_i < y_{i+1}, \\
Q_X(x, x + e_i) & \text{ if } (x', y') = (x + e_i, y + e_{i+1}) \text{ and } x_i = y_{i+1}, \\
Q_X(x, x - e_i) & \text { if } (x', y') = (x - e_i, y) \text{ and } x_i > y_i, \\
Q_X(x, x - e_i) & \text{ if } (x', y') = (x - e_i, y - e_i) \text{ and } x_i = y_i, \\
v_{n+1}^{\pm} & \text{ if } (x', y') = (x, y \pm e_i)
\end{cases}
\end{equation*}
with $Q_X$ given by the $Q$-matrix from Proposition \ref{GT_process}.
All other off-diagonal entries are zero and 
the diagonal entries $\mathcal{A}((x', y'), (x', y'))$ equal the negative of
\begin{equation}
\label{diagonal_entries}
 \sum_{i=1}^n v_i + \sum_{i=1}^n v_i^{-1} + \sum_{i=1}^n v_{n+1} 1_{\{y_i' < x_i'\}}
+ \sum_{i=1}^n v_{n+1}^{-1} 1_{\{y_{i+1}' > x_i'\}} + v_{n+1} + v_{n+1}^{-1}.
\end{equation}
The inductive hypothesis that $Q_X$ is conservative means that $\mathcal{A}$ is conservative.
We define the function \begin{equation*}
m(x, y) = v_{n+1}^{\lvert y \rvert - \lvert x \rvert} \frac{S_x(v)}{S_y(v)}, \text{ for } (x, y) \in W^{n, n+1} 
\end{equation*}
and an intertwining kernel given by $\Lambda : W^{n+1} \rightarrow W^{n, n+1}$
\begin{equation*}
\Lambda(y, (x', y')))  = m(x', y') 1_{\{y = y'\}}.
\end{equation*}
The key step in proving Proposition \ref{GT_process} is to prove the intertwining 
$Q_Y \Lambda = \Lambda \mathcal{A}$ where $Q_Y$ is the desired $Q$-matrix from 
Proposition \ref{GT_process} with $n+1$ particles. This is equivalent to the statement that
\begin{equation}
\label{intertwining}
Q_{Y}(y, y') = \sum_{x \leq y} \frac{m(x, y)}{m(x', y')} \mathcal{A}((x, y), (x', y')), \text{ for all }y, y' \in W^n
\end{equation}
Once this is established it follows from general theory \cite{rogers1981} that $Y$ is an autonomous
Markov process 
with the desired $Q$-matrix and this $Q$-matrix is conservative; therefore Proposition \ref{GT_process} follows inductively. 
For a more detailed argument we refer to \cite{warren_windridge}: we are replacing Equation 3.3 from \cite{warren_windridge} 
with equation \eqref{intertwining} and the rest of the argument is unchanged. 
It remains to show \eqref{intertwining}.

We first show \eqref{intertwining} in the case $y' = y$. The right hand side of \eqref{intertwining} equals
\begin{equation*}
\sum_{x \preceq y} v_{n+1}^{\lvert x' \rvert - \lvert x \rvert} \frac{S_x(v)}{S_{x'}(v)} \mathcal{A}((x, y'), (x', y'))
\end{equation*}
and $\mathcal{A}((x, y'), (x', y'))$ can be non-zero if $x = x'$ or $x = x' \pm e_i$. 
If $x = x'$ then $\mathcal{A}((x', y'), (x', y'))$ equals the negative of \eqref{diagonal_entries}.
If $x = x' - e_i$ then $\mathcal{A}((x, y'), (x', y')) = Q_X(x'-e_i, x') 1_{\{y_{i}' < x_{i}'\}}$ and
if $x = x' + e_i$ then $\mathcal{A}((x, y'), (x', y')) = Q_X(x'+e_i, x') 1_{\{y_{i+1}' > x_{i}'\}}$. 
Using these expressions we find that the right hand side of \eqref{intertwining}
is equal to 
\begin{multline*}
-\left(\sum_{i=1}^n v_i + \sum_{i=1}^n v_i^{-1} + \sum_{i=1}^n v_{n+1} 1_{\{ y_i' < x_i' \}} 
+ \sum_{i=1}^n v_{n+1}^{-1} 1_{\{y_{i+1}' > x_i'\}} + v_{n+1} + v_{n+1}^{-1} \right)\\
+ \sum_{i=1}^n v_{n+1} \frac{S_{x' - e_i}(v)}{S_{x'}(v)} \frac{S_{x'}(v)}{S_{x' - e_i}(v)} 1_{\{y_i' < x_i'\}}
+ \sum_{i=1}^n v_{n+1}^{-1} \frac{S_{x' + e_i}(v)}{S_{x'}(v)} \frac{S_{x'}(v)}{S_{x' + e_i}(v)} 1_{\{y_{i+1}' > x_{i}'\}}
\end{multline*}
This equals $-\sum_{i=1}^{n+1} v_i - \sum_{i=1}^{n+1} v_i^{-1} = Q_Y(y', y')$ and proves \eqref{intertwining} for 
$y' = y$.

Suppose that $y' = y+e_i$ and consider two cases: depending on whether or not there was a pushing interaction. 
If $i = 1$ or $ i > 1$ and $x_{i-1}' < y_i'$, the right hand side of \eqref{intertwining} equals
\begin{equation*}
\frac{v_{n+1}^{\lvert y' - e_i \rvert - \lvert x' \rvert}}{v_{n+1}^{\lvert y' \rvert - \lvert x' \rvert}} 
\frac{S_{x'}(v)}{S_{y' - e_i}(v)} \frac{S_{y'}(v)}{S_{x'}(v)} v_{n+1} = Q_Y(y' - e_i, y').
\end{equation*}
The second case is if $i > 1$ and $x_{i-1}' = y_i'$, when the right hand side of \eqref{intertwining} equals
\begin{equation*}
\frac{v_{n+1}^{\lvert y' - e_i \rvert - \lvert x' - e_i\rvert}}{v_{n+1}^{\lvert y' \rvert - \lvert x' \rvert}} \frac{S_{x'-e_i}(v)}{S_{y' - e_i}(v)}
\frac{S_{y'}(v)}{S_{x'}(v)}
\frac{S_{x'}(v)}{S_{x' - e_i}(v)} = Q_Y(y' - e_i, y').
\end{equation*}
Finally suppose that $y' = y - e_i$ and split again into two cases. If $i = n$ or $i < n $ and $y_i' < x_i'$ then the right hand side of \eqref{intertwining} equals
\begin{equation*}
\frac{v_{n+1}^{\lvert y' + e_i \rvert - \lvert x' \rvert}}{v_{n+1}^{\lvert y' \rvert - \lvert x' \rvert}} 
\frac{S_{x'}(v)}{S_{y' + e_i}(v)} \frac{S_{y'}(v)}{S_{x'}(v)} v_{n+1}^{-1} = Q_Y(y'+e_i, y') 
\end{equation*} 
If $i < n$ and $y_i' = x_i'$ 
then the right hand side 
of \eqref{intertwining} equals
\begin{equation*}
\frac{v_{n+1}^{\lvert y' + e_i \rvert - \lvert x' + e_i\rvert}}{v_{n+1}^{\lvert y' \rvert - \lvert x' \rvert}} \frac{S_{x' + e_i}(v)}{S_{y' + e_i}(v)}
\frac{S_{y'}(v)}{S_{x'}(v)}
\frac{S_{x'}(v)}{S_{x' + e_i}(v)} = Q_Y(y' + e_i, y').
\end{equation*}
This completes the proof of \eqref{intertwining} and as described above completes the proof of Proposition \ref{GT_process}. 

\begin{proof}[Proof of Proposition \ref{two_sided_intertwining}]
We construct the process $\mathfrak{X}$ in 
Proposition \ref{GT_process} started from the origin with the rate of right and left jump rates on level $j$ given by $v_{n-j+1}$  
and $v_{n-j+1}^{-1}$ respectively.
Part (i) is an immediate consequence. 
Part (ii) 
follows from the fact that $(\mathfrak{X}^n_n(t))_{t \geq 0}$ is the 
largest particle in two different Markov processes $(\mathfrak{X}^n_1(t), \mathfrak{X}^n_2(t), \ldots, \mathfrak{X}^n_n(t))_{t \geq 0}$ which has $Q$-matrix given in Proposition \ref{GT_process}
and $(\mathfrak{X}^1_1(t), \mathfrak{X}^2_2(t), \ldots, \mathfrak{X}^n_n(t))_{t \geq 0}$ which
is PushASEP (without a wall) where the $i$-th particle has right jump rate $v_{n-i+1}$ and left jump rate $v_{n-i+1}$.  
The top particle of PushASEP (without a wall) is equal in distribution as a process to the right hand side of Proposition \ref{two_sided_intertwining}
by the argument used to prove equation \eqref{semi_discrete_lpp}.
\end{proof}

\paragraph{Acknowledgements.} I am very grateful to Jon Warren for helpful and stimulating discussions and to Neil O'Connell 
for suggesting the approach in Section 2.1. I am grateful for the financial support of the Royal Society Enhancement Award `Log-correlated Gaussian fields and symmetry classes in random matrix theory RGF\textbackslash EA\textbackslash 181085.'


\end{document}